\documentclass[final,onefignum,onetabnum,letterpaper]{siamonline190516}
\usepackage[scale=0.8]{geometry} 

\usepackage{lipsum}
\usepackage{amsfonts}
\usepackage{epstopdf}
\ifpdf
  \DeclareGraphicsExtensions{.eps,.pdf,.png,.jpg}
\else
  \DeclareGraphicsExtensions{.eps}
\fi

\usepackage{datetime}
\newdateformat{monthyeardate}{%
  \monthname[\THEMONTH] \THEDAY, \THEYEAR}

\usepackage{academicons}
\usepackage{xcolor}
\renewcommand{\orcid}[1]{\href{https://orcid.org/#1}{\textcolor[HTML]{A6CE39}{orcid.org/#1}}}

\usepackage{amsmath}
\allowdisplaybreaks
\usepackage{amssymb}
\usepackage{commath}
\usepackage{mathtools}
\usepackage{bbm}

\usepackage{color}
\usepackage{graphicx}
\usepackage[small]{caption}
\usepackage{subcaption}

\usepackage{relsize}
\usepackage{adjustbox}
\usepackage{algorithm}
\usepackage[noend]{algpseudocode}
\usepackage{booktabs}
\usepackage{tikz}

\usepackage{verbatim}

\usepackage{enumitem}
\setlist[enumerate]{leftmargin=.5in}
\setlist[itemize]{leftmargin=.5in}

\newsiamthm{problem}{Problem}
\newsiamremark{remark}{Remark}
\newsiamremark{hypothesis}{Hypothesis} 
\crefname{hypothesis}{Hypothesis}{Hypotheses}
\newsiamremark{example}{Example}
\newsiamthm{claim}{Claim}
\newsiamthm{conjecture}{Conjecture}

\headers{SBP operators for general function spaces}{J.\ Glaubitz, J.\ Nordstr\"om, and P.\ \"Offner}

\title{Summation-by-parts operators for general function spaces
\thanks{
\monthyeardate\today 
\corresponding{Jan Glaubitz} 
}}

\author{
Jan Glaubitz\thanks{Department of Mathematics, Dartmouth College, Hanover, NH 03755, USA (\email{Jan.Glaubitz@Dartmouth.edu}, \orcid{0000-0002-3434-5563})}
\and 
{Jan Nordstr\"om\thanks{Department of Mathematics, Link\"oping University, 58183, Link\"oping, Sweden (\email{jan.nordstrom@liu.se}, \orcid{0000-0002-7972-6183})}} \thanks{Department of Mathematics and Applied Mathematics, University of Johannesburg, P.\,O.\ Box 524, Auckland Park 2006, Johannesburg, South Africa}
\and 
Philipp \"Offner\thanks{Institute of Mathematics, Johannes Gutenberg University, Mainz, Germany, (\email{poeffner@uni-mainz.de}, \orcid{0000-0002-1367-1917})} 
}

\usepackage{amsopn}

\DeclareMathOperator{\diag}{diag}

\newcommand{\Span}{\mathrm{span}}

\renewcommand{\d}{\mathrm{d}}
\newcommand{\intd}{\, \mathrm{d}}

\newcommand{\R}{\mathbb{R}}

\usepackage{lineno}

\ifpdf
\hypersetup{
  pdftitle={Glaubitz2021FSBP},
  pdfauthor={Jan Glaubitz}
}
\fi

\begin{document}

\maketitle

\begin{abstract}
Summation-by-parts (SBP) operators are popular building blocks for systematically developing stable and high-order accurate numerical methods for time-dependent differential equations.  
The main idea behind existing SBP operators is that the solution is assumed to be well approximated by polynomials up to a certain degree, and the SBP operator should therefore be exact for them.  
However, polynomials might not provide the best approximation for some problems, and other approximation spaces may be more appropriate. 
In this paper, a theory for SBP operators based on general function spaces is developed. 
We demonstrate that most of the established results for polynomial-based SBP operators carry over to this general class of SBP operators. 
Our findings imply that the concept of SBP operators can be applied to a significantly larger class of methods than currently known. 
We exemplify the general theory by considering trigonometric, exponential and radial basis functions. 
\end{abstract}




\section{Introduction} 
\label{sec:introduction} 

In combination with a weak enforcement of boundary conditions (BCs), SBP operators allow for a systematic development of energy-stable semi-discretizations of energy-bounded initial boundary value problems (IBVPs) \cite{svard2014review,fernandez2014review}. 
Initially, the SBP concept was developed in the 1970s for finite difference (FD) methods \cite{kreiss1974finite,kreiss1977existence,scherer1977energy}. 
A summary of the accumulated theory for SBP operators as of 1994 is provided in \cite{strand1994summation}. 
More recent reviews can be found in \cite{svard2014review} (with a focus on theoretical properties of SBP operators) and \cite{fernandez2014review} (with a focus on the construction of SBP operators) as well as references therein. 
The works on SBP-based finite element methods are too numerous to list entirely, but some examples include spectral element \cite{carpenter2014entropy}, discontinuous Galerkin \cite{gassner2013skew,chen2017entropy}, finite volume \cite{nordstrom2001finite,nordstrom2003finite}, and flux reconstruction  \cite{huynh2007flux,ranocha2016summation,offner2018stability} schemes.
Also see \cite{nordstrom2019energy,abgrall2020analysis,abgrall2021analysis} and references therein. 
Other applications of SBP operators include essentially non-oscillatory (ENO) and weighted ENO (WENO) schemes \cite{carpenter2016entropy,fisher2011boundary,yamaleev2009systematic}, and implicit time integration methods \cite{nordstrom2013summation,linders2020properties,ranocha2021new}.

Although some of their details might differ, at their core, existing SBP operators are constructed to be exact for polynomials up to a certain degree. 
This corresponds to the---sometimes implicit---underlying assumption that, given a certain partial differential equation (PDE), polynomials are suited to approximate its solutions. 
However, for some IBVPs, polynomials might not be the best choice. 
Instead, other approximation spaces should be used. 
To illustrate this, consider boundary layer or highly-oscillatory periodic solutions. 
These solutions can be expected to be better approximated by exponential and trigonometric functions, respectively. 
Indeed, \cref{fig:expl_approx} provides a rudimentary demonstration of the advantage of using other approximation spaces instead of polynomial ones for these two examples. 

\begin{figure}[tb]
	\centering 
	\begin{subfigure}[b]{0.49\textwidth}
		\includegraphics[width=\textwidth]{%
      		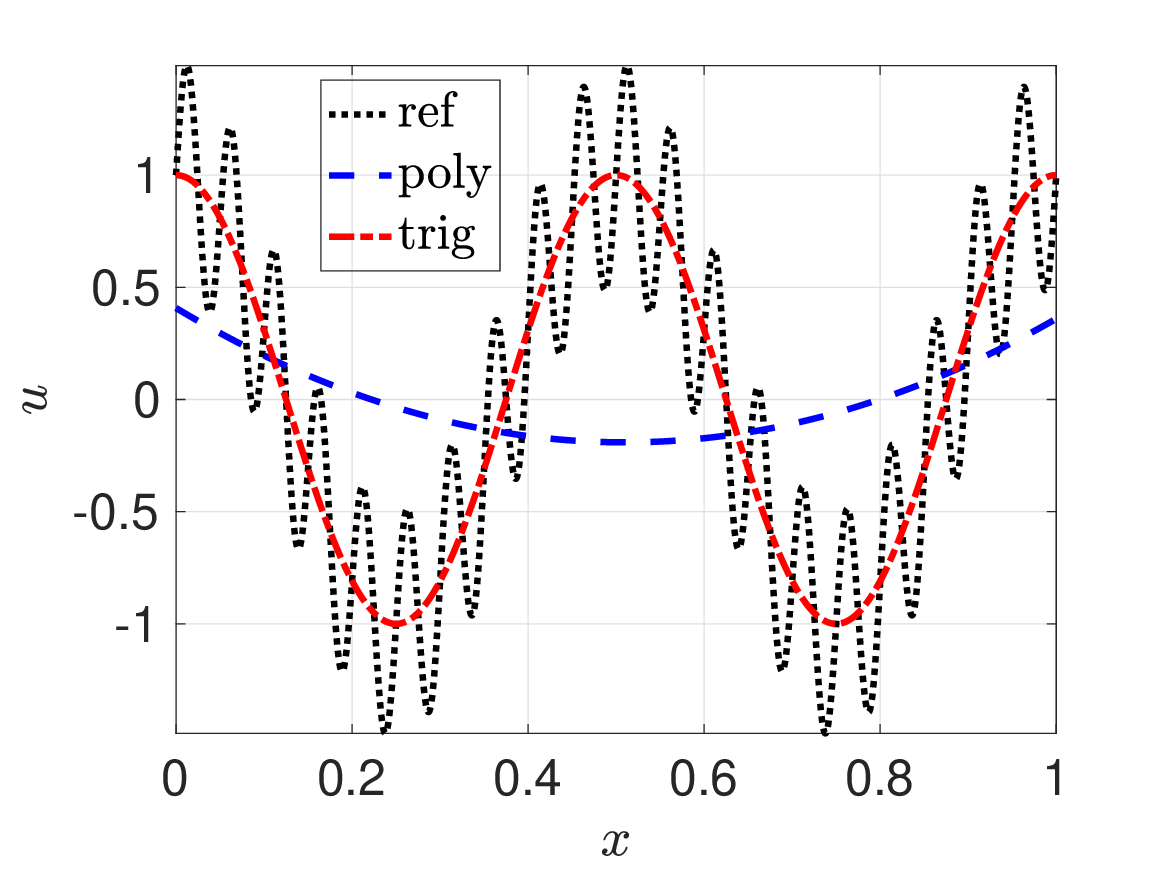} 
    		\caption{Highly oscillatory solution}
    		\label{fig:expl_oscillatory_approx}
  	\end{subfigure}%
	~
  	\begin{subfigure}[b]{0.49\textwidth}
		\includegraphics[width=\textwidth]{%
      		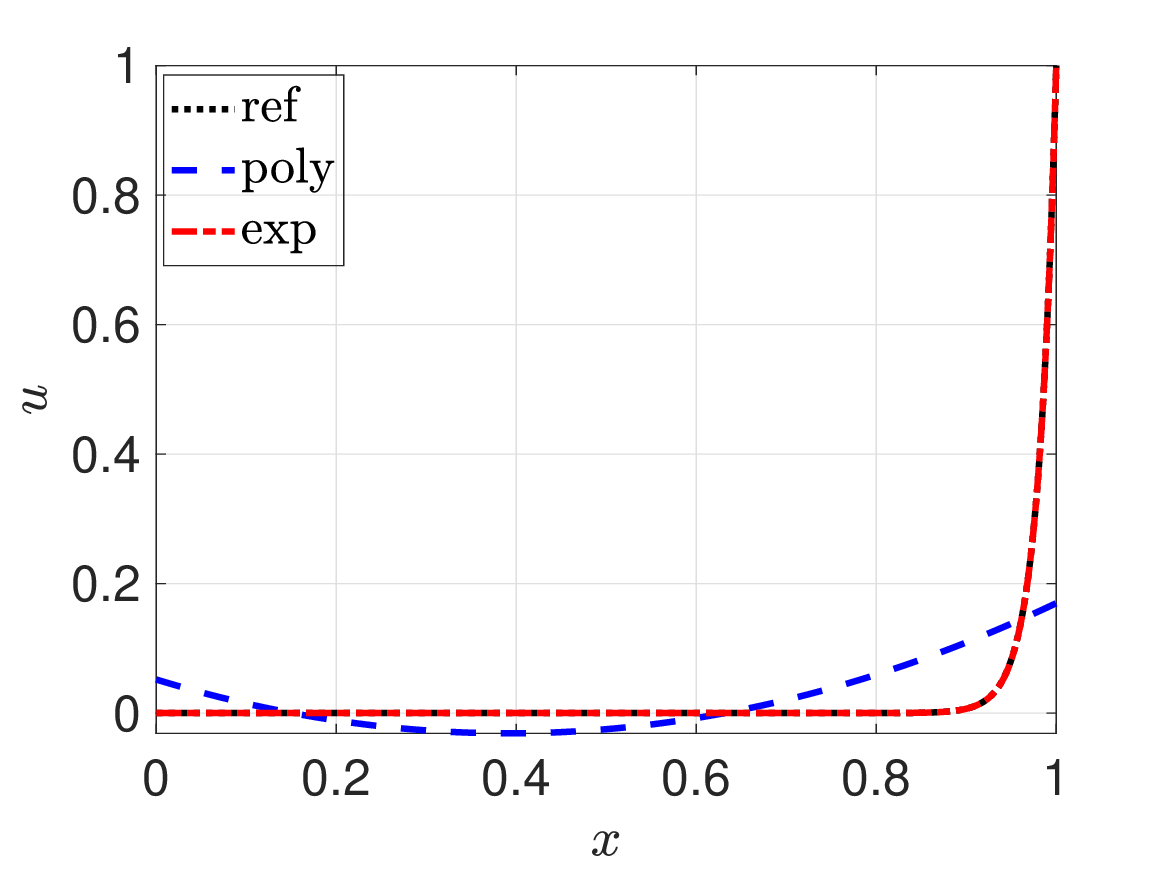} 
    		\caption{Boundary layer solution}
    		\label{fig:expl_layer_approx}
  	\end{subfigure}%
  	\caption{
  	Polynomial (``poly"), trigonometric (``trig") and exponential (``exp") least-squares function approximations to a boundary layer function and a highly oscillatory solution. 
	The polynomial, trigonometric and exponential approximation space was $\Span \{ 1,x, x^2 \}$, $\Span \{ 1,\sin( 4\pi x ), \cos( 4\pi x ) \}$ and $\Span \{ 1, x, e^{\alpha x} \}$ with suitably chosen $\alpha \in \R$, respectively. 
  	}
  	\label{fig:expl_approx}
\end{figure}  

The potential advantages of using non-polynomial approximation spaces have been pointed out in several previous works. 
For instance, in \cite{kadalbajoo2003exponentially,kalashnikova2009discontinuous}, exponentially fitted schemes were used to solve singular perturbation problems. 
Discontinuous Galerkin methods based on non-polynomial approximation spaces were considered in \cite{yuan2006discontinuous}. 
There are also several works on ENO and WENO reconstructions based on non-polynomial function approximations \cite{christofi1996study,iske1996structure,hesthaven2019entropy}. 
Another line of related work deals with global and local radial basis function methods \cite{fasshauer1996solving,fornberg2015solving,fornberg2015primer} and flux reconstruction methods based on radial basis functions \cite{watson2022RBF}. 
Finally, there are some recent efforts on numerical methods for PDEs based on rational function approximations \cite{nakatsukasa2018aaa,gopal2019solving}. 

In this paper, we develop a theory for SBP operators based on general approximation spaces. 
It is demonstrated that most of the established results for polynomial-based SBP operators carry over to this general class of SBP operators. 
This includes different mimetic properties, the connection between certain quadrature formulas and SBP operators, and strategies for their construction. 
It is also revealed that the general approximation space has to include constants for the SBP operator to mimic conservation. 
Our findings imply that the concept of SBP operators can be applied to a larger class of methods than currently known. 
In particular, SBP operators for general approximation spaces in combination with a weak enforcement of boundary conditions \cite{glaubitz2021stabilizing,glaubitz2021towards} will enable us to construct energy-stable radial basis function methods, which we will elaborate on in a forthcoming work. 

The rest of this work is organized as follows. 
In \cref{sec:SBP_operators}, we introduce the concept of SBP operators for general function spaces, which will be referred to as FSBP operators. 
In \cref{sec:properties}, we address mimetic properties of these FSBP operators, such as conservation and energy-stability. 
The connection between FSBP operators and certain positive and exact quadrature formulas is discussed in \cref{sec:connection}. 
Building up on this connection, \cref{sec:construction} focuses on the construction FSBP operators. 
In \cref{sec:examples}, we present a few examples of FSBP operators for trigonometric, exponential, and radial basis function spaces. 
Finally, we end this work with a series of numerical tests in \cref{sec:num_tests} and some concluding thoughts in \cref{sec:summary}.  
\section{SBP operators for general approximation spaces} 
\label{sec:SBP_operators} 

Here, we describe how the concept of SBP operators on the interval $[x_L,x_R]$ can be extended to general approximation spaces.

\subsection{SBP operators for polynomial approximation spaces}
\label{sub:classic_SBP}

SBP operators are constructed to be mimetic of the integration by parts (IBP) property. 
At their core, SBP operators are usually defined as follows. 

\begin{definition}[SBP operators]\label{def:SBP_usual}
	An operator $D = P^{-1} Q$ is a (polynomial-based) \emph{SBP operator} of degree $d$ if 
	\begin{enumerate}
		\item[(i)] 
		$D \mathbf{x}^j = j \mathbf{x}^{j-1}$ for $j=0,\dots,d$, 
		
		\item[(ii)] 
		$P$ is a symmetric positive definite matrix, and 
		
		\item[(iii)] 
		$Q + Q^T = B := \diag(-1,0,\dots,0,1)$.
		
	\end{enumerate} 
	Here, $\mathbf{x}^j = [x_1^j, \dots, x_N^j]^T$ denotes the vector of the function values of the $j$-th monomial $x^j$ evaluated at the grid points $x_1,\dots,x_N$. 
\end{definition}

It should be noted that \cref{def:SBP_usual} includes some restrictions, which can be relaxed. 
For instance, we assume that the grid points $\mathbf{x}$ include the boundary points. 
This restriction was relaxed, for instance, in \cite{fernandez2014generalized} and analyzed in \cite{nordstrom2017conservation,offner2019error}. 
Further generalizations include upwind operators \cite{mattsson2017diagonal}, higher derivatives \cite{mattsson2004summation,mattsson2014diagonal,ranocha2021broad}, curvilinear or staggered grids \cite{o2017energy,aalund2019encapsulated,fernandez2019staggered}, and multiple dimensions \cite{hicken2016multidimensional}.
Also the combination of SBP operators with splitting techniques \cite{nordstrom2006conservative,gassner2016split} to construct entropy stable methods should be mentioned; see \cite{chen2020review} and references therein.  
However, for clarity of presentation, we will base our generalization of SBP operators on \cref{def:SBP_usual}. 

\begin{remark}[Norm matrix $P$]
	We restrict the norm matrix $P$ in \cref{def:SBP_usual} to be symmetric and positive definite. 
	The latter is needed for $P$ to induce a discrete norm, ${\|\mathbf{u}\|_P^2 = \mathbf{u}^T P \mathbf{u}}$ for $\mathbf{u} \in \R^N$, and the first to obtain energy estimates such as the one in \cref{sub:linear_energy}. 
	See the reviews \cite{svard2014review,fernandez2014review} and references therein. 
\end{remark}

\subsection{SBP operators for general function spaces} 
\label{sub:gen_SBP}

While this is not always obvious, the high-level idea behind (i) in \cref{def:SBP_usual} is that polynomials of degree up to $d$ are expected to approximate the PDE solution well for $d$ high enough, and the differentiation operator $D$ should therefore be exact for them. 
That is, (i) in \cref{def:SBP_usual} can be reformulated as 
\begin{equation}\label{eq:classic_SBP_exactness}
	D f(\mathbf{x}) = f'(\mathbf{x}) \quad \forall f \in \mathcal{P}_d.
\end{equation} 
Here, $\mathcal{P}_d$ denotes the linear space of polynomials of degree up to $d$ and $f(\mathbf{x}), f'(\mathbf{x})$ respectively denote the vector of the function values of $f$ and its derivative $f'$ at the grid points $\mathbf{x}$: 
\begin{equation} 
	f(\mathbf{x}) = [f(x_1), \dots, f(x_N)]^T, \quad 
	f'(\mathbf{x}) = [f'(x_1), \dots, f'(x_N)]^T. 
\end{equation}
It is now clearly possible to replace $\mathcal{P}_d$ in \cref{eq:classic_SBP_exactness} by a non-polynomial function space with other approximation properties.  
To this end, assume that it is reasonable to approximate the solution $u$ with functions from a finite-dimensional approximation space $\mathcal{F} \subset C^1([x_L,x_R])$. 
In this case, \cref{eq:classic_SBP_exactness} is modified to
\begin{equation} 
	D f(\mathbf{x}) = f'(\mathbf{x}) \quad \forall f \in \mathcal{F}.
\end{equation}
Consequently, it is natural to formulate the following generalized definition of SBP operators.

\begin{definition}[FSBP operators]\label{def:SBP_general} 
	Let $\mathcal{F} \subset C^1([x_L,x_R])$ be a finite-dimensional function space. 
	An operator $D = P^{-1} Q$ is an \emph{$\mathcal{F}$-based SBP operator} if 
	\begin{enumerate}
		\item[(i)] \label{item:SBP_general1} 
		$D f(\mathbf{x}) =f'(\mathbf{x})$ for all $f \in \mathcal{F}$, 
		
		\item[(ii)] \label{item:SBP_general2} 
		$P$ is a symmetric positive definite matrix, and 
		
		\item[(iii)] \label{item:SBP_general3} 
		$Q + Q^T = B = \diag(-1,0,\dots,0,1)$.
		
	\end{enumerate}  
\end{definition} 

Note that only (i) in \cref{def:SBP_general} differs from \cref{def:SBP_usual}. 
Henceforth, for sake of simplicity, we will refer to the $\mathcal{F}$-based SBP operators in \cref{def:SBP_general} as \emph{FSBP operators}. 

We end this section by noting that \cref{def:SBP_general,def:SBP_usual} only require the norm matrix $P$ to be symmetric positive definite. 
That said, for simplicity, we henceforth restrict ourselves to diagonal norm matrices $P$.  
\section{Mimetic properties of FSBP-SAT methods} 
\label{sec:properties} 

Here, we address some basic mimetic properties of FSBP operators. 
This serves two purposes: 
\begin{itemize} 
	\item It is shown that for conservation it is necessary to include constants in the function space $\mathcal{F}$. 

	\item It is demonstrated that most other mimetic results for polynomial-based SBP operators carry over to the more general framework of FSBP operators. 
	
\end{itemize} 
For clarity, we illustrate this for the linear advection equation.   
Similar mimetic properties hold also for other equations, which is briefly addressed in \cref{app:outlook}. 

\subsection{The linear advection equation} 
\label{sub:linear} 

Consider the linear advection equation 
\begin{equation}\label{eq:linear}
	\partial_t u + a \partial_x u = 0, \quad x \in (x_L,x_R), \ t>0,
\end{equation} 
with $a > 0$. 
Moreover, \cref{eq:linear} is equipped with suitable initial and boundary conditions, 
\begin{equation}\label{eq:IBC}
	u(0,x) = u_0(x), \quad 
	u(t,x_L) = g(t).
\end{equation} 
Given an FSBP operator $D = P^{-1}Q$, the FSBP-SAT semi-discretization of \cref{eq:linear}, \cref{eq:IBC} is formally given by 
\begin{equation}\label{eq:lin_discr}
	\mathbf{u}_t  +a D \mathbf{u} = P^{-1} \mathbb{S}.
\end{equation} 
Here, $\mathbf{u} = [u_1,u_2,\dots,u_N]^T$ denotes the vector of the nodal values of the numerical solution at the grid points $\mathbf{x} = [x_1,\dots,x_N]^T$.   
Furthermore, $\mathbb{S}$ is a \emph{simultaneous approximation term (SAT)} that weakly enforces the boundary condition \cref{eq:IBC}. 
The idea behind adding this term is to \emph{simultaneously approximate} the equation and the boundary condition by ``pulling" the numerical solution towards the boundary data $g$ at the left boundary. 
For \cref{eq:linear} with positive wave speed, the SAT is given by 
\begin{equation} 
	\mathbb{S} = [\mathbb{S}_1,0,\dots,0]^T, \quad 
	\mathbb{S}_1 = - \sigma a (u_1-g),
\end{equation}
where the parameter $\sigma$ remains to be determined. 
The subsequent analysis of FSBP-SAT methods follows the one for polynomial-based SBP operators \cite{svard2014review,fernandez2014review}.

\subsection{Conservation} 
\label{sub:linear_cons} 

Recall that the exact solution of \cref{eq:linear} satisfies 
\begin{equation}\label{eq:lin_cons}
	\frac{\d}{\d t} \int_{x_L}^{x_R} u \intd x = - a \left[ u(t,x_R) - g(t) \right].
\end{equation} 
This means that the total amount of the quantity $u$ (e.\,g.\ mass) is neither created nor destroyed inside the domain, and only changes due to the flux across the boundaries. 
This is referred to as \emph{conservation}. 
In many applications, conservation is an important physical property and is therefore desired to be mimicked by the numerical scheme. 
Note that given an appropriate norm matrix $P$, the left-hand side of \cref{eq:lin_cons} can be approximated by $\mathbf{1}^T P \mathbf{u}_t$.
Then, substituting the FSBP-SAT semi-discretization \cref{eq:lin_discr}, we get 
\begin{equation}\label{eq:semi_SAT}
	\mathbf{1}^T P \mathbf{u}_t 
		= - a \mathbf{1}^T PD \mathbf{u} + \mathbf{1}^T \mathbb{S}. 
\end{equation} 
The SBP properties transform \cref{eq:semi_SAT} to  
\begin{equation}\label{eq:semi_SAT3}
	\mathbf{1}^T P \mathbf{u}_t 
		= - a \left[ u_N - \left( \sigma g(t) + [ 1-\sigma ] u_1 \right) \right] + a (D \mathbf{1})^T P \mathbf{u}. 
\end{equation} 
The first term on the right-hand side of \cref{eq:semi_SAT3} approximates the right-hand side of \cref{eq:lin_cons}. 
Note that $\sigma g(t) + (1-\sigma) u_1$ is a linear combination of the boundary data $g(t)$ and the value of the numerical solution at the left boundary $u_1$. 
However, \cref{eq:semi_SAT3} also contains the additional volume term $a (D \mathbf{1})^T P \mathbf{u}$, while no such term is present in \cref{eq:lin_cons}. 
To avoid artificial construction or destruction of the quantity $u$, $D \mathbf{1} = \mathbf{0}$ must hold. 
For polynomial-based SBP operators this is automatically satisfied. 
For FSBP-operators, on the other hand, we can ensure this by requiring the function space $\mathcal{F}$ to contain constants (see (i) in \cref{def:SBP_general}). 
In this case, we get 
\begin{equation}\label{eq:lin_discr_cons_sigma} 
	\mathbf{1}^T P \mathbf{u}_t 
		= - a \left[ u_N - \left( \sigma g(t) + [ 1-\sigma ] u_1 \right) \right]. 
\end{equation} 
In particular, for $\sigma = 1$, \cref{eq:lin_discr_cons_sigma} yields 
\begin{equation}\label{eq:lin_discr_cons}
	\mathbf{1}^T P \mathbf{u}_t = - a [ u_N - g(t) ], 
\end{equation} 
which is a discrete analogue to \cref{eq:lin_cons}. 
Note that other choices for $\sigma$ would be possible, but the choice $\sigma = 1$ comes with the additional advantage of error minimization,  
see \cref{app:error}.

\subsection{Energy-stability} 
\label{sub:linear_energy} 

Besides being conservative, exact solutions of the linear advection equation \cref{eq:linear} are also energy-stable. 
That is, the growth of their energy over time is bounded as demonstrated by the energy method: 
\begin{equation}\label{eq:cont_energy_rate} 
	\frac{\d}{\d t} \| u \|_{L^2}^2 
		= - a \int_{x_L}^{x_R} 2 u \partial_x u \intd x 
		= -a \left[  u(t,x_R)^2 - u(t,x_L)^2 \right] 
		\leq a g^2(t). 
\end{equation} 
In order to establish stability for the numerical solution, a similar bound to \cref{eq:cont_energy_rate} is required. 
To this end, we observe that the left-hand side of \cref{eq:cont_energy_rate} can be approximated by $\frac{\d}{\d t} \mathbf{u}^T P \mathbf{u}$ on a discrete level. 
Furthermore, we note that 
\begin{equation}\label{eq:lin_rate}
	\frac{\d}{\d t} \mathbf{u}^T P \mathbf{u} 
		= 2 \mathbf{u}^T P \mathbf{u}_t \\
		= -2a \mathbf{u}^T P D \mathbf{u} + 2 \mathbf{u}^T \mathbb{S}, 
\end{equation} 
where \cref{eq:lin_discr} has been used.
The SBP properties therefore yield
\begin{equation}\label{eq:lin_rate2}
	\frac{\d}{\d t} \|\mathbf{u}\|^2_P 
		= a \left( u_1^2 - u_N^2 - 2 \sigma u_1^2 + 2 \sigma u_1 g \right) 
		\leq ag^2 \left( \frac{\sigma^2}{2\sigma-1} \right) 
\end{equation}
for $\sigma > 1/2$, where $\|\mathbf{u}\|^2_P = \mathbf{u}^T P \mathbf{u}$.
See \cite{abgrall2020analysis,fernandez2014review,svard2014review} for more details. 
Clearly \cref{eq:lin_rate2} is a discrete analog to \cref{eq:cont_energy_rate}. 
The upper bound $a g^2 \sigma^2/(2\sigma-1)$ is minimized for $\sigma =1$, and becomes the right-hand side of \cref{eq:cont_energy_rate}. 
Also see \cite{aalund2016provably} and \cref{app:error}. 

We end this section by summarizing the results so far.  

\begin{theorem}\label{thm:lin_stable}
	Let $\mathcal{F} \subset C^1$ be a finite-dimensional function space and $D$ be an $\mathcal{F}$-based SBP operator. 
	The FSBP-SAT semi-discretization \cref{eq:lin_discr} of the linear advection equation \cref{eq:linear} is
	\begin{itemize} 
		\item 
		conservative if $\mathcal{F}$ contains constants. 
		
		\item 
		energy-stable if $\sigma > 1/2$.	
	\end{itemize}
\end{theorem} 

\begin{remark} 
For energy-stability, it is not necessary to include constants in the function space $\mathcal{F}$. 
This is only needed to mimic conservation on a discrete level. 
\end{remark} 
 
\section{FSBP operators and quadrature formulas} 
\label{sec:connection} 

We now investigate the connection between FSBP operators and certain quadrature formulas. 
It is again demonstrated that most results for polynomial-based SBP operators \cite{hicken2016multidimensional,svard2014review,fernandez2014review} carry over to FSBP operators.
In particular, it is shown that the existence of FSBP operators is equivalent to the existence of certain quadrature formulas.

\subsection{Quadrature formulas} 
\label{sub:connection_QFs} 

We start by providing a few comments on exactness conditions of quadrature formulas \cite{engels1980numerical,davis2007methods,brass2011quadrature}. 
Let $I_{\mathbf{x}, \mathbf{w}}$ be a quadrature formula consisting of a set of points $\mathbf{x} = [x_1,\dots,x_N]^T$ and weights $\mathbf{w} = [w_1,\dots,w_N]^T$ such that 
\begin{equation} 
	I_{\mathbf{x}, \mathbf{w}}[f] 
		:= \sum_{n=1}^N w_n f(x_n) 
		\approx \int_{x_L}^{x_R} f(x) \intd x 
		=: I[f].
\end{equation}
Here, $f:[x_L,x_R] \to \R$ is a suitable (continuous) function, and we assume that the points $\mathbf{x}$ all lie in $[x_L,x_R]$.

\begin{definition}\label{def:QF_pos_exact}  
	We say that a quadrature formulas $I_{\mathbf{x}, \mathbf{w}}$ is \emph{positive} if its weights are positive, i.\,e., $w_n > 0$ for all $n=1,\dots,N$. 
	Moreover, $I_{\mathbf{x}, \mathbf{w}}$ is said to be \emph{$\mathcal{G}$-exact} if the following \emph{exactness condition} is satisfied:
	\begin{equation} 
		I_{\mathbf{x}, \mathbf{w}}[g] = I[g] \quad \forall g \in \mathcal{G}.
	\end{equation} 
\end{definition}

We already pre-empt that this function space will differ from the approximation space $\mathcal{F}$ for which the SBP operator $D$ is exact. 
Indeed, $\mathcal{F}$-based diagonal-norm SBP operator will be connected to positive and $\mathcal{G}$-exact quadrature formulas, where $\mathcal{G} = (\mathcal{F}^2)'$ is of higher dimension than $\mathcal{F}$. 
More on this in \cref{sub:connection_char}.

\subsection{A necessary and sufficient condition for the existence of FSBP operators} 
\label{sub:connection_char} 

Our main result (\cref{cor:diag_char} below) shows that a diagonal-norm FSBP operator $D$ exists if and only if there exists a positive and $(\mathcal{F}\mathcal{F})'$-exact quadrature formula, both supported on the same grid points. 
Here, $(\mathcal{F}\mathcal{F})'$ denotes the space of all functions that correspond to the derivative of a product of two functions from $\mathcal{F} \subset C^1([x_L,x_R])$. 
That is, 
\begin{equation} 
	(\mathcal{F} \mathcal{F})' 
		= \left\{ \, (f g)' \mid f,g \in \mathcal{F} \, \right\}.
\end{equation}
Note that, by the chain rule, $(f g)' = f'g + g'f$. 

\begin{example}\label{expl:F_poly}
	For polynomial-based SBP operators, we have $\mathcal{F} = \mathcal{P}_d$. 
	In this case, $(\mathcal{F} \mathcal{F})'$ is equal to the space of polynomials of degree up to $2d-1$, $\mathcal{P}_{2d-1}$. 
\end{example}

We start by noting that the norm matrix $P$ corresponding to an FSBP operator (diagonal or nondiagonal) has to satisfy a certain compatibility relation. 

\begin{lemma}[Necessary condition for general FSBP operators]\label{lem:gen_nec} 
	If $D = P^{-1} Q$ is an $\mathcal{F}$-based SBP operator, then the norm matrix $P$ satisfies the \emph{compatibility relation} 
	\begin{equation}\label{eq:comp_cond}
		g(\mathbf{x})^T P f'(\mathbf{x}) + f(\mathbf{x})^T P g'(\mathbf{x})
			= fg \big|_{x_L}^{x_R} 
			\quad \forall f,g \in \mathcal{F}. 
	\end{equation}
\end{lemma}

\begin{proof} 
	\cref{lem:gen_nec} is well-known in the context of polynomial-based SBP operators. 
	See, for instance, \cite{hicken2013summation,linders2018order} or the reviews \cite{svard2014review,fernandez2014review}. 
	It is easy to verify that the arguments in the proof for the polynomial case also hold for general function spaces.  
\end{proof} 

\cref{lem:gen_nec} provides us with a necessary condition for $D = P^{-1} Q$ to be an $\mathcal{F}$-based SBP operator in terms of the norm matrix $P$. 
Indeed, if $P$ is not satisfying the compatibility relation \cref{eq:comp_cond}, then $D = P^{-1} Q$ cannot be an $\mathcal{F}$-based SBP operator. 
We point out that \cref{lem:gen_nec} holds for diagonal- as well as nondiagonal-norm matrices $P$. 
Henceforth, we restrict ourselves to FSBP operators with diagonal-norm matrix $P$.
This restriction allows us to formulate a sharper necessary condition for the existence of FSBP operators. 

\begin{theorem}[Necessary condition for diagonal-norm FSBP operators]\label{thm:diag_nec} 
	Let $D = P^{-1} Q$ be an $\mathcal{F}$-based SBP operator with positive definite diagonal-norm matrix $P = \diag(\mathbf{p})$. 
	Then the diagonal elements of $P$ correspond to the weights of a positive and $(\mathcal{F} \mathcal{F})'$-exact quadrature formula. 
\end{theorem} 

\begin{proof} 
	\cref{lem:gen_nec} implies that $P$ has to satisfy the compatibility equation \cref{eq:comp_cond}. 
	For a diagonal-norm matrix $P = \diag(p_1,\dots,p_N)$, \cref{eq:comp_cond} becomes 
	\begin{equation}\label{eq:diag_nec_proof1}
		\sum_{n=1}^N p_n \left[ f'(x_n) g(x_n)  + g'(x_n) f(x_n)  \right]
			= fg \big|_{x_L}^{x_R} 
			\quad \forall f,g \in \mathcal{F}.
	\end{equation} 
	Next note that $f'g + g'f = (fg)'$, and that $fg$ on the right-hand side of \cref{eq:diag_nec_proof1} is the anti-derivative of this function. 
	Thus, \cref{eq:diag_nec_proof1} can be reformulated as 
	\begin{equation}\label{eq:diag_nec_proof2} 
		\sum_{n=1}^N p_n (fg)'(x_n) = \int_{x_L}^{x_R} (fg)'(x) \intd x \quad \forall f,g \in \mathcal{F}.
	\end{equation} 
	Finally, \cref{eq:diag_nec_proof2} is the exactness condition of a positive and $(\mathcal{F} \mathcal{F})'$-exact quadrature formula with weights $\mathbf{p} = [p_1,\dots,p_N]^T$; see \cref{def:QF_pos_exact}. 
\end{proof}

\cref{thm:diag_nec} gives a necessary condition for the existence of a diagonal-norm FSBP operator. 
This necessary condition is the existence of a positive and $(\mathcal{F} \mathcal{F})'$-exact quadrature formula. 
\cref{thm:diag_suf} below reveals that this condition is not just necessary but also sufficient. 
Before presenting this theorem, we introduce the concept of Vandermonde-like matrices. 
Let $\{f_1,\dots,f_K\}$ be a basis of the $K$-dimensional function space $\mathcal{F} \subset C^1$. 
Evaluating the basis functions at the grid points $\mathbf{x}$ and writing these function values in a matrix, we get the Vandermonde-like matrix 
\begin{equation}\label{eq:Vandermonde}
	F = [f_1(\mathbf{x}), \dots, f_K(\mathbf{x})] =
		\begin{bmatrix} 
			f_1(x_1) & \dots & f_K(x_1) \\ 
			\vdots & & \vdots \\ 
			f_1(x_N) & \dots & f_K(x_N)
		\end{bmatrix}.
\end{equation}
\cref{thm:diag_suf} will require this matrix to have linearly independent columns.\footnote{%
$F$ having linearly independent columns is equivalent to the grid points $\mathbf{x}$ being $\mathcal{F}$-unisolvent, which is an important concept in the context of multivariate approximation theory \cite{wendland2004scattered,fornberg2015primer,glaubitz2020shock} and numerical integration \cite{glaubitz2020stableCFs,glaubitz2021constructing,glaubitz2021towardsRBF}.} 
This will ensure that $F$ can be converted to an invertible matrix by appending an appropriate set of vectors (assuming $K \leq N$).

\begin{theorem}[Sufficient condition for diagonal-norm FSBP operators]\label{thm:diag_suf} 
	Let $\mathcal{F} \subset C^1$ and assume that the matrix $F$ in \cref{eq:Vandermonde} has linearly independent columns. 
	If there is a positive and $(\mathcal{F} \mathcal{F})'$-exact quadrature formula $I_{\mathbf{x}, \mathbf{p}}$, then there exists an $\mathcal{F}$-based SBP operator $D = P^{-1}Q$ with a positive diagonal-norm matrix given by $P = \diag(\mathbf{p})$. 
\end{theorem} 

\begin{proof} 
	It is easy to verify that the same arguments as in \cite[Proof of Theorem 3.3]{hicken2016multidimensional}, where polynomial-based SBP operators were addressed, can be used. 
\end{proof}

Combining \cref{thm:diag_nec,thm:diag_suf}, we are now able to formulate the following characterization of existence of diagonal-norm FSBP operators. 

\begin{corollary}\label{cor:diag_char} 	
	Let $\mathcal{F} \subset C^1$ be a finite-dimensional function space and assume that the matrix $F$ in \cref{eq:Vandermonde} has linearly independent columns.
	Then there exists an $\mathcal{F}$-based SBP operator $D=P^{-1}Q$ with a positive definite diagonal-norm matrix $P=\diag(\mathbf{p})$ if and only if there exists a positive and $(\mathcal{F} \mathcal{F})'$-exact quadrature formula $I_{\mathbf{x}, \mathbf{p}}$. 
\end{corollary}

\begin{proof} 
	\cref{cor:diag_char} consists of two parts: 
	(i) If there exists an $\mathcal{F}$-based SBP operator $D=P^{-1}Q$ with positive definite diagonal-norm matrix $P$, then there also exists a positive and $(\mathcal{F} \mathcal{F})'$-exact quadrature formula supported on the same grid points. 
	Moreover, the weights of this quadrature formula are given by the positive diagonal elements of $P$. 
	This part was proved by \cref{thm:diag_nec}. 
	(ii) If there exists a positive and $(\mathcal{F} \mathcal{F})'$-exact quadrature formula and the Vandermonde-like matrix $F$ in \cref{eq:Vandermonde} has linearly independent columns, then there also exists an $\mathcal{F}$-based SBP operator $D=P^{-1}Q$ with positive definite diagonal-norm matrix $P$. 
	Furthermore, the FSBP operator $D=P^{-1}Q$ is defined on the same grid points as the quadrature formula and the diagonal-norm matrix $P$ contains the positive quadrature weights as its diagonal elements. 
\end{proof}

We end this section by putting \cref{cor:diag_char} into the context of classical poly\-nomial-based SBP operators. 

\begin{remark} 
	\cref{cor:diag_char} is well-known for polynomial-based SBP operators \cite{svard2014review,fernandez2014review}. 
	Recall from \cref{expl:F_poly} that for $\mathcal{F} = \mathcal{P}_d$, we have $(\mathcal{F}\mathcal{F})' = \mathcal{P}_{2d-1}$. 
	Hence, in this case, \cref{cor:diag_char} states that a $\mathcal{P}_d$-based SBP operator (usually called an SBP operator of degree $d$) exists if and only if there is a positive and $\mathcal{P}_{2d-1}$-exact quadrature formula (usually called a quadrature formula of degree $2d-1$). 
	There are many such quadrature formulas. 
	Yet, the minimal number of grid points (including the boundary points), given by $N = d+1$, is achieved by the Gauss--Lobatto formula. 
\end{remark}
 
\section{Construction of FSBP operators and associated quadrature formulas} 
\label{sec:construction} 

We again restrict the discussion to diagonal-norm FSBP operators on a grid $\mathbf{x} \in \R^N$ in $[x_L,x_R]$ with $x_1 = x_L$ and $x_N = x_R$. 
Furthermore, we assume that the grid points $\mathbf{x}$ yield a Vandermonde-like matrix $F$ in \cref{eq:Vandermonde} with linearly independent columns.

\subsection{Construction of FSBP operators}
\label{sub:constr_FSBP}

To start, we assume that we have already found a positive and $(\mathcal{F}\mathcal{F})'$-exact quadrature formula $I_{\mathbf{x}, \mathbf{p}}$. 
Then we can construct a diagonal-norm $\mathcal{F}$-based SBP operator $D = P^{-1} Q$ by using the strategy proposed in \cite{hicken2016multidimensional} for polynomial-based SBP operators. 
For completeness, we revisit the most important steps. 
First, choose 
\begin{equation} 
	P = \diag(\mathbf{p}). 
\end{equation}
Next, note that by decomposing $Q$ into its symmetric and anti-symmetric part, $Q_S$ and $Q_A$, (iii) in \cref{def:SBP_general} yields 
\begin{equation}\label{eq:Q_QA} 
	Q = Q_A + \frac{1}{2} B.
\end{equation}
Furthermore, the accuracy condition (i) in \cref{def:SBP_general} implies 
\begin{equation}\label{eq:exactness_QA} 
	Q_A F = P F_x - \frac{1}{2} B F.
\end{equation}
Here, $F_x$ is the Vandermonde-like matrix for the derivatives of the basis elements $f_1,\dots,f_K$ of $\mathcal{F}$, defined analogously to \cref{eq:Vandermonde} as ${F_x = [f_1'(\mathbf{x}), \dots, f_K'(\mathbf{x})]}$. 
It remains to construct an anti-symmetric $Q_A$ that satisfies \cref{eq:exactness_QA}. 
Once such a $Q_A$ is found, we get $Q$ by \cref{eq:Q_QA} and the FSBP operator as $D = P^{-1} Q$. 
To this end, \cref{eq:exactness_QA} can be recast as a linear system 
\begin{equation}\label{eq:LS_q} 
 	A \mathbf{q} = \mathbf{y}, 
\end{equation}
where $\mathbf{q}$ denotes the vector that contains the strictly lower part of $Q_A$: 
\begin{equation} 
	q\left( \frac{[i-2][i-1]}{2} + j \right) = (Q_A)_{i,j}, \quad 
	2 \leq i \leq n, \ 1 \leq j < i.
\end{equation}
The strictly upper part of $Q_A$ is then obtained according to $(Q_A)_{j,i} = - (Q_A)_{i,j}$, and the diagonal elements are set to zero. 
Note that \cref{eq:LS_q} consists of $NK$ equations for $(N-1)N/2$ unknowns. 
Although this might result in more equations than unknowns (if $K > (N-1)/2$), \cref{thm:diag_suf} ensures that there exists \emph{at least one solution}. 
Indeed, \cref{eq:LS_q} can have infinitely many solutions. 
In this case, the remaining degrees of freedom can be used to optimize the SBP operator based on different criteria, such as the truncation error and spectral properties \cite{strand1994summation,svard2014review,fernandez2014review}. 
In our implementation, we select the unique least-squares solution \cite{golub2012matrix} with minimal Euclidean norm $\|\cdot\|_{\ell^2}$ among all possible solutions of \cref{eq:LS_q} since this is computationally convenient.

\subsection{Construction of positive and exact quadrature formulas}
\label{sub:constr_QF}

We start by noting that for polynomial-based SBP operators, with $\mathcal{F} = \mathcal{P}_d$, $( \mathcal{F}\mathcal{F} )'$-exact quadrature formulas are the ones which have degree of exactness $2d-1$ or larger. 
Positive quadrature formulas with degree of exactness $2d-1$ (that include the boundary points) are given, for instance, by 
(1) Gauss--Lobatto formulas, which come with a minimal number of $N = d$ grid points; 
(2) compound Newton--Cotes formulas, which can be used on equidistant grids, but are positive only for a limited degree of exactness; and 
(3) least-squares formulas, which can be used on almost arbitrary sets of grid points and any degree of exactness, but with $N$ being proportional to $d^2$, rather than $d$ as for (1), making them less efficient.   
See one of the monographs \cite{gautschi1997numerical,davis2007methods,brass2011quadrature} for (1) and (2), and \cite{huybrechs2009stable,glaubitz2020stableDG,glaubitz2020stableQF} for (3). 
Besides polynomial spaces, quadrature theory might also be considered as well-developed for some other function spaces, such as trigonometric functions \cite{mysovskikh2001cubature}, exponential functions \cite{cordero1996gaussian}, and radial basis functions \cite{sommariva2021rbf,glaubitz2021towardsRBF}. 
For instance, the composite trapezoidal rule gives a positive and $(\mathcal{F} \mathcal{F})'$-exact quadrature formula on equidistant points for trigonometric function spaces $\mathcal{F}$. 
For all other function spaces considered here we were able to use the least-squares approach to construct positive and $(\mathcal{F} \mathcal{F})'$-exact quadrature formulas. 
It was proven in \cite{glaubitz2021constructing} that a positive and $(\mathcal{F} \mathcal{F})'$-exact least-squares quadrature is found whenever (i) $(\mathcal{F} \mathcal{F})'$ includes constants and (ii) sufficiently many grid points are used. 
More precisely, it was numerically observed in \cite{glaubitz2021constructing} that for a positive and $(\mathcal{F} \mathcal{F})'$-exact least-squares quadrature, $N$ has to be proportional to $L^2$ with $L$ being the dimension of $(\mathcal{F} \mathcal{F})'$. 
The proof presented in \cite{glaubitz2021constructing} is not restricted to equidistant points, but also allows for other classes of equidistributed points, including low-discrepancy points used in quasi-Monte Carlo methods. 
Moreover, random points were addressed in \cite{migliorati2022stable}.  
\section{Some examples} 
\label{sec:examples}

For illustrative purposes, we subsequently address a few special cases of non-polynomial approximation spaces.

\subsection{Trigonometric approximation spaces}

Consider the trigonometric approximation space 
\begin{equation}\label{eq:trig_space}
	\mathcal{T}_d 
		= \Span\{ \, 1, \sin( k \omega x ), \cos( k \omega x ) \mid k=1,\dots,d \, \}
\end{equation} 
with frequency $\omega = 2 \pi [x_R - x_L]$. 
The dimension of this approximation space is ${K = 1 + 2d}$. 
Using elemental product-to-sum identities \cite[4.3.31--33]{abramowitz1972handbook}, the product space $\mathcal{F}\mathcal{F}$ and its derivative space $(\mathcal{F}\mathcal{F})'$ are respectively given by 
\begin{equation} 
	\mathcal{T}_d \mathcal{T}_d = \mathcal{T}_{2d}, \quad 
	(\mathcal{T}_d \mathcal{T}_d)' = \mathcal{T}_{2d} \setminus \Span\{1\}.
\end{equation} 
Consequently, $\mathcal{T}_d \mathcal{T}_d$ has dimension $4d+1$ and $(\mathcal{T}_d \mathcal{T}_d)'$ has dimension $4d$, since the constant has been removed.
It is worth noting that, in particular,  
\begin{equation}\label{eq:DFF_vs_FF_trig} 
	(\mathcal{T}_d \mathcal{T}_d)' \subset \mathcal{T}_{2d}. 
\end{equation} 
It now remains to find a positive and $(\mathcal{T}_d \mathcal{T}_d)'$-exact quadrature formula. 
To this end, it is convenient to recall that $\mathcal{T}_{2d}$ is exactly integrated by the  composite trapezoidal rule using $N=2d+1$ equidistant points \cite[Chapter 3.2.1]{gautschi1997numerical}. 
In particular, because of \cref{eq:DFF_vs_FF_trig}, the composite trapezoidal rule using $N=2d+1$ equidistant points gives us a positive and $(\mathcal{T}_d \mathcal{T}_d)'$-exact quadrature formula. 
That said, we still must use a larger number of $N = 2d+2$ equidistant grid points to construct the corresponding FSBP operator, since $F$ in \cref{eq:Vandermonde} only has linearly independent columns for $N > 2d+1$, which is required for the sufficient condition \cref{thm:diag_suf} to hold. 

On $[x_L,x_R] = [0,1]$ and for $d=1$, the trigonometric approximation space \cref{eq:trig_space} is given by ${\mathcal{T}_1 = \Span\{ \, 1, \sin( \omega x ), \cos( \omega x ) \, \}}$ and has dimension $K=3$. 
Using $N=K+1=4$ equidistant grid points, the composite trapezoidal rule is positive and $(\mathcal{T}_1\mathcal{T}_1)'$-exact. 
The grid points and quadrature weights are 
\begin{equation} 
\begin{aligned}
	\mathbf{x} = \left[0, \frac{1}{3}, \frac{2}{3}, 1 \right]^T, \quad 
	\mathbf{w} = \left[ \frac{1}{6}, \frac{1}{3}, \frac{1}{3}, \frac{1}{6} \right]^T,
\end{aligned}
\end{equation} 
The corresponding matrices $Q$ and $D$ of the FSBP operator $D = P^{-1} Q$ obtained from the construction procedure described in \cref{sub:constr_FSBP} are  
\begin{equation} 
\renewcommand*{\arraystretch}{1.4} 
	Q \approx
	\begin{bmatrix} 
		-\frac{1}{2} & \frac{3}{5} & -\frac{3}{5} & \frac{1}{2}\\ -\frac{3}{5} & 0 & \frac{121}{100} & -\frac{3}{5}\\ \frac{3}{5} & -\frac{121}{100} & 0 & \frac{3}{5}\\ -\frac{1}{2} & \frac{3}{5} & -\frac{3}{5} & \frac{1}{2} 
	\end{bmatrix}, \quad 
	D \approx
	\begin{bmatrix} 
		-3 & \frac{363}{100} & -\frac{363}{100} & 3\\ -\frac{181}{100} & 0 & \frac{363}{100} & -\frac{181}{100}\\ \frac{181}{100} & -\frac{363}{100} & 0 & \frac{181}{100}\\ -3 & \frac{363}{100} & -\frac{363}{100} & 3 
	\end{bmatrix}.
\end{equation}
Here, we have rounded the numbers to the second decimal place.

\subsection{Exponential approximation spaces}
\label{sub:exp_FSBP}

Consider the exponential approximation space 
\begin{equation}\label{eq:exp_space}
	\mathcal{E}_d 
		= \mathcal{P}_{d-1} \cup \Span\{ \, e^{x} \, \} 
		= \Span\{ \, x^k, e^{x} \mid k=0,1,\dots,d-1 \, \} 
\end{equation} 
with dimension ${K = d + 1}$. 
The product space $\mathcal{F}\mathcal{F}$ and its derivative space $(\mathcal{F}\mathcal{F})'$ are respectively given by 
\begin{equation} 
\begin{aligned} 
	\mathcal{E}_d \mathcal{E}_d 
		& = \Span\left\{ \, x^k, x^l e^{x}, e^{2x} \mid k=0,\dots,2d-2,\ l=0,\dots,d-1 \, \right\}, \\
	(\mathcal{E}_d \mathcal{E}_d)' 
		& = \Span\left\{ \, x^k, x^l e^{x}, e^{2x} \mid k=0,\dots,2d-3,\ l=0,\dots,d-1 \, \right\}. 
\end{aligned}
\end{equation} 
Consequently, $\mathcal{E}_d \mathcal{E}_d$ has dimension $3d$ and $(\mathcal{E}_d \mathcal{E}_d)'$ has dimension $3d-1$.
It now remains to find a positive and $(\mathcal{E}_d \mathcal{E}_d)'$-exact quadrature formula. 
We were able to find such quadrature formulas using the least-squares approach \cite{huybrechs2009stable,glaubitz2021constructing} on equidistant points. 

On $[x_L,x_R] = [0,1]$ and for $d=2$, the three-dimensional exponential approximation space \cref{eq:exp_space} is given by ${\mathcal{E}_2 = \Span\{ \, 1, x, e^{x} \, \}}$. 
Using $N=5$ equidistant grid points, we found the least-squares quadrature formula with the following points and weights to be positive and $(\mathcal{E}_2\mathcal{E}_2)'$-exact:
\begin{equation} 
\begin{aligned}
	\mathbf{x} = \left[0, \frac{1}{4}, \frac{1}{2}, \frac{3}{4}, 1 \right]^T, \quad 
	\mathbf{w} \approx \left[ \frac{2}{25}, \frac{9}{25}, \frac{3}{25}, \frac{9}{25}, \frac{2}{25} \right]^T,
\end{aligned}
\end{equation} 
where we have again rounded the numbers to the second decimal place. 
The corresponding matrices $Q$ and $D$ of the FSBP operator $D = P^{-1} Q$ obtained from the construction procedure described in \cref{sub:constr_FSBP} are  
\begin{equation} 
\renewcommand*{\arraystretch}{1.3} 
\resizebox{.9\textwidth}{!}{$\displaystyle 
	Q \approx
	\begin{bmatrix} 
		-\frac{1}{2} & \frac{13}{20} & -\frac{1}{25} & -\frac{19}{100} & \frac{7}{100}\\ -\frac{13}{20} & 0 & \frac{8}{25} & \frac{13}{25} & -\frac{19}{100}\\ \frac{1}{25} & -\frac{8}{25} & 0 & \frac{8}{25} & -\frac{1}{25}\\ \frac{19}{100} & -\frac{13}{25} & -\frac{8}{25} & 0 & \frac{13}{20}\\ -\frac{7}{100} & \frac{19}{100} & \frac{1}{25} & -\frac{13}{20} & \frac{1}{2} 
	\end{bmatrix}, \ \  
	D \approx
	\begin{bmatrix} 
		-\frac{329}{50} & \frac{859}{100} & -\frac{23}{50} & -\frac{127}{50} & \frac{49}{50}\\ -\frac{9}{5} & 0 & \frac{22}{25} & \frac{29}{20} & -\frac{53}{100}\\ \frac{7}{25} & -\frac{257}{100} & 0 & \frac{129}{50} & -\frac{29}{100}\\ \frac{53}{100} & -\frac{29}{20} & -\frac{89}{100} & 0 & \frac{181}{100}\\ -\frac{49}{50} & \frac{249}{100} & \frac{12}{25} & -\frac{213}{25} & \frac{653}{100} 
	\end{bmatrix}, 
$}
\end{equation}
where we have again rounded the numbers to the second decimal place.

\subsection{Radial basis functions}

Consider the three-dimensional function space 
\begin{equation} 
	\mathcal{C}(X_3) 
		= \Span\{ \, c_1, c_2, c_3 \, \} 
\end{equation} 
on $[x_L,x_R] = [0,1]$ with 
\begin{equation}\label{eq:cardinal_functions}
\begin{aligned}
	c_1(x) 
		& = \frac{1}{2} |x|^3 - 2 | x - 1/2 |^3 + \frac{3}{2} |x-1|^3 - \frac{1}{4}, \\ 
	c_2(x) 
		& = -2 |x|^3 + 4 | x - 1/2 |^3 - 2 |x-1|^3 + \frac{3}{2}, \\ 
	c_3(x) 
		& = \frac{3}{2} |x|^3 - 2 | x - 1/2 |^3 + \frac{1}{2} |x-1|^3 - \frac{1}{4}. 
\end{aligned} 
\end{equation} 
The function space $\mathcal{C}(X_3) $ corresponds to the linear space spanned by the cubic radial basis function (RBF) interpolants using the centers $X_3 = \{0,1/2,1\}$ augmented with a constant. 
The basis elements $c_1$, $c_2$, and $c_3$ in \cref{eq:cardinal_functions} are the associated cardinal functions.
The product space $\mathcal{C}(X_3)\mathcal{C}(X_3)$ and its derivative space $(\mathcal{C}(X_3)\mathcal{C}(X_3))'$ are respectively given by 
\begin{equation} 
\begin{aligned} 
	\mathcal{C}(X_3) \mathcal{C}(X_3) 
		& = \Span\{ \, c_1, c_2, c_3, c_1^2, c_2^2, c_3^2, c_1 c_2, c_1 c_3, c_2 c_3 \, \} \\ 
	(\mathcal{C}(X_3) \mathcal{C}(X_3))' 
		& = \Span\{ \, c_1', c_2', c_3', (c_1^2)', (c_2^2)', (c_3^2)', (c_1 c_2)', (c_1 c_3)', (c_2 c_3)' \, \}. 
\end{aligned}
\end{equation} 
Using $N=4$ equidistant grid points, we found the least-squares quadrature formula (see \cite{huybrechs2009stable,glaubitz2021constructing}) with the following points and weights to be positive and $(\mathcal{C}(X_3) \mathcal{C}(X_3))'$-exact:
\begin{equation} 
\begin{aligned}
	\mathbf{x} = \left[0, \frac{1}{3}, \frac{2}{3}, 1 \right]^T, \quad 
	\mathbf{w} = \left[ \frac{16}{129}, \frac{81}{215}, \frac{81}{215}, \frac{16}{129} \right]^T.
\end{aligned}
\end{equation} 
The corresponding matrices $Q$ and $D$ of the FSBP operator $D = P^{-1} Q$ obtained from the construction procedure described in \cref{sub:constr_FSBP} are  
\begin{equation} 
\renewcommand*{\arraystretch}{1.4} 
	Q \approx
	\begin{bmatrix} 
		-\frac{1}{2} & \frac{59}{100} & -\frac{3}{20} & \frac{3}{50}\\ -\frac{59}{100} & 0 & \frac{37}{50} & -\frac{3}{20}\\ \frac{3}{20} & -\frac{37}{50} & 0 & \frac{59}{100}\\ -\frac{3}{50} & \frac{3}{20} & -\frac{59}{100} & \frac{1}{2}  
	\end{bmatrix}, \quad 
	D \approx
	\begin{bmatrix} 
		-\frac{403}{100} & \frac{473}{100} & -\frac{121}{100} & \frac{51}{100}\\ -\frac{39}{25} & 0 & \frac{49}{25} & -\frac{2}{5}\\ \frac{2}{5} & -\frac{49}{25} & 0 & \frac{39}{25}\\ -\frac{51}{100} & \frac{121}{100} & -\frac{473}{100} & \frac{403}{100} 
	\end{bmatrix}, 
\end{equation}
where we have again rounded the numbers to the second decimal place. 
While the above example only discusses cubic RBF interpolants, we have investigated the construction of FSBP operators for more general RBFs in \cite{glaubitz2022energy}. 
\section{Numerical results} 
\label{sec:num_tests} 

We used the explicit strong stability preserving (SSP) Runge--Kutta (RK) method of third order using three stages (SSPRK(3,3)) \cite{shu1988total} for all subsequent numerical tests. 
For simplicity, we only consider FSBP operators on equidistant grid points. 
That said, while these results are not reported here, we are also able to construct FSBP operators on nonequidistant points. 
We report on some of these results in \cite{glaubitz2022energy}, where FSBP operators on random and Halton points are considered. 
The MATLAB code used to generate the numerical tests presented here is open access and can be found on GitHub.\footnote{See \url{https://github.com/jglaubitz/FSBP}}

\subsection{Homogeneous linear advection equation}

Consider the homogeneous linear advection equation 
\begin{equation}\label{eq:linear_adv} 
\begin{aligned} 
	\partial_t u + \partial_x u & = 0, \quad && 0 < x < 1, \\ 
	u(x,0) & = \cos(4\pi x) + 0.5 \sin(40 \pi x), \quad && 0 \leq x \leq 1, \\ 
	u(0,t) & = u(1,t), \quad && t \geq 0,
\end{aligned}
\end{equation} 
with periodic boundary conditions and highly oscillatory initial data. 

\begin{figure}[tb]
	\centering 
	\begin{subfigure}[b]{0.49\textwidth}
		\includegraphics[width=\textwidth]{%
      		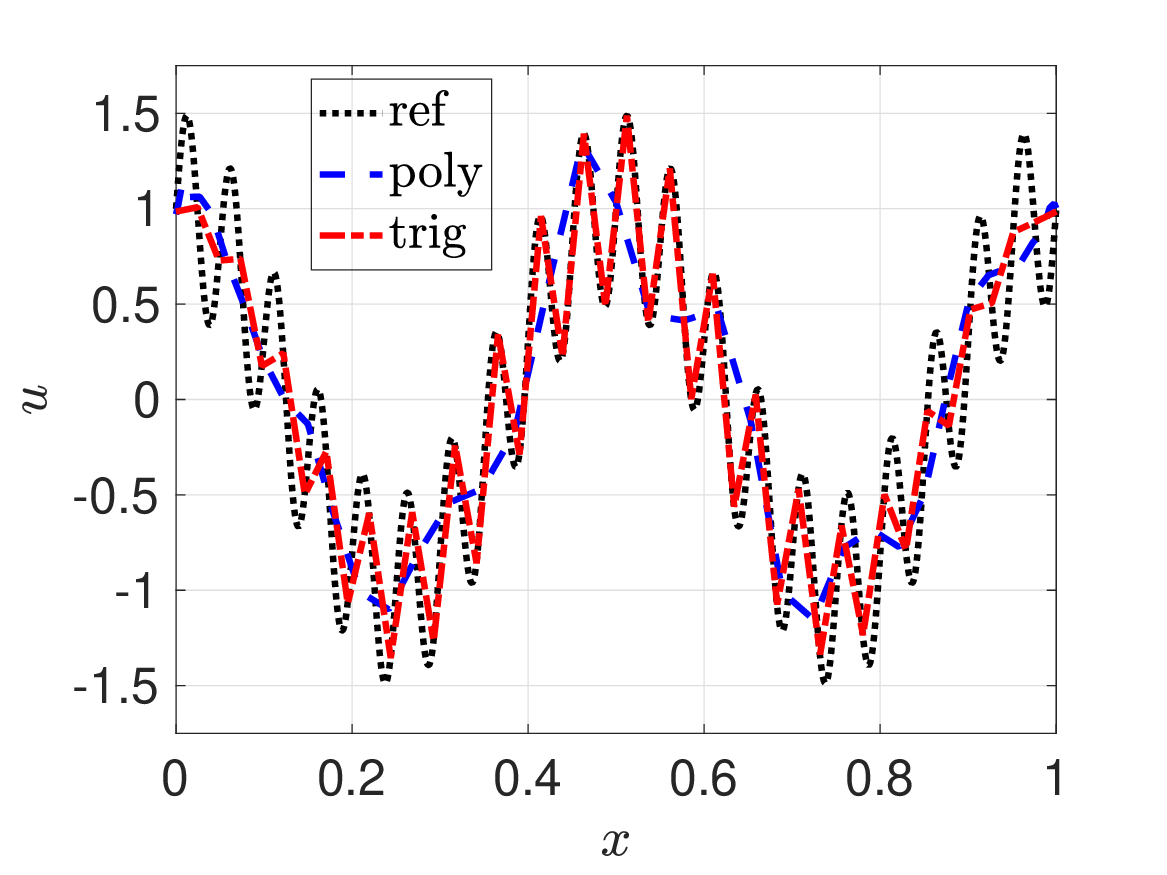} 
    		\caption{Solutions for $K=41$}
    		\label{fig:linear_sol_K41}
  	\end{subfigure}%
	~
  	\begin{subfigure}[b]{0.49\textwidth}
		\includegraphics[width=\textwidth]{%
      		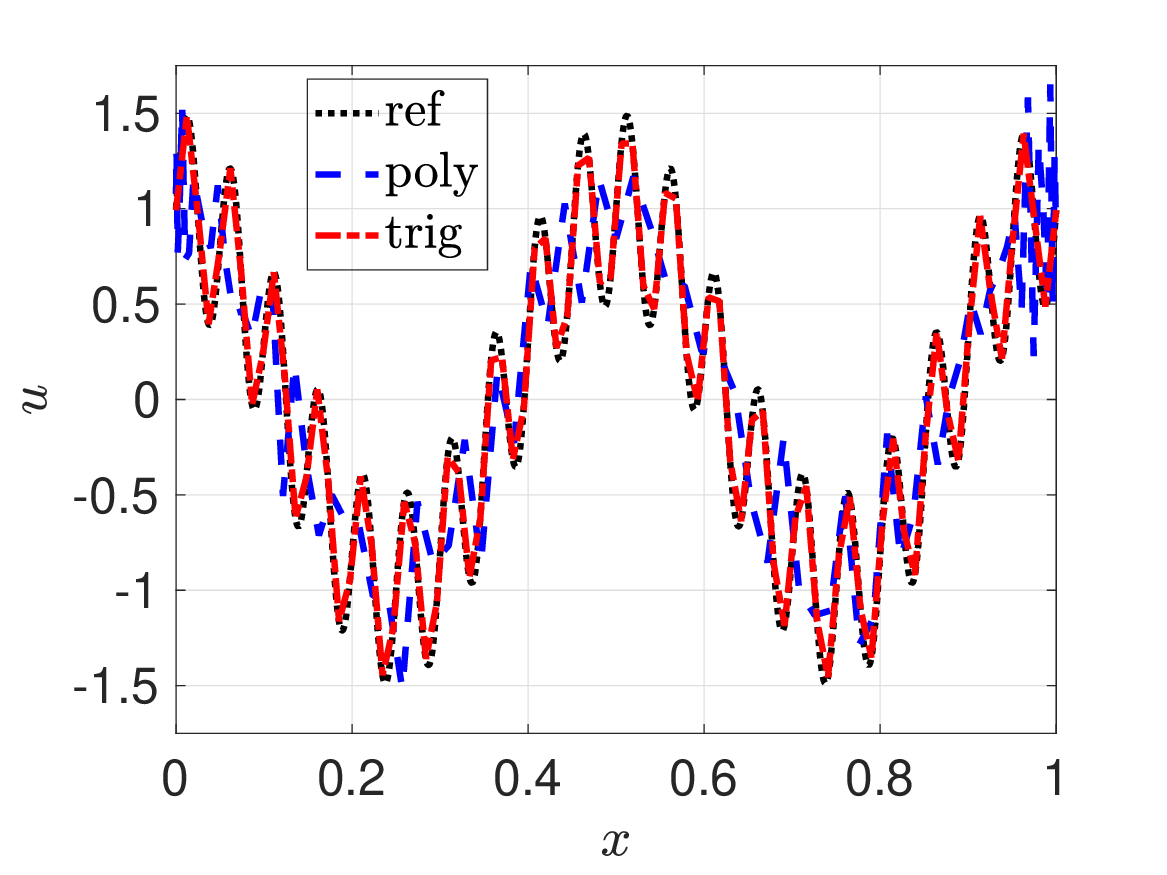} 
    		\caption{Solutions for $K=81$}
    		\label{fig:linear_sol_K81}
  	\end{subfigure}%
	\\ 
	\begin{subfigure}[b]{0.49\textwidth}
		\includegraphics[width=\textwidth]{%
      		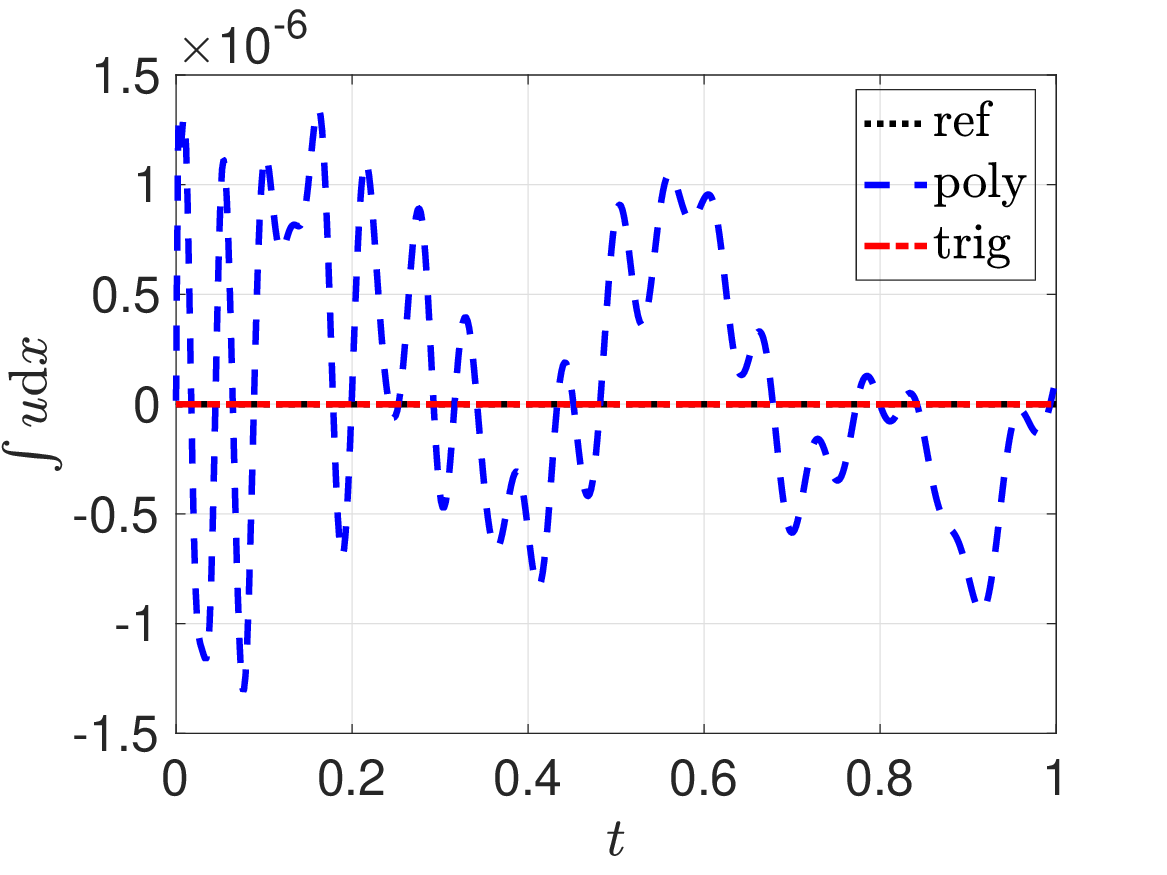} 
    		\caption{Mass over time for $K=81$}
    		\label{fig:linear_mass_K81}
  	\end{subfigure}%
	~
  	\begin{subfigure}[b]{0.49\textwidth}
		\includegraphics[width=\textwidth]{%
      		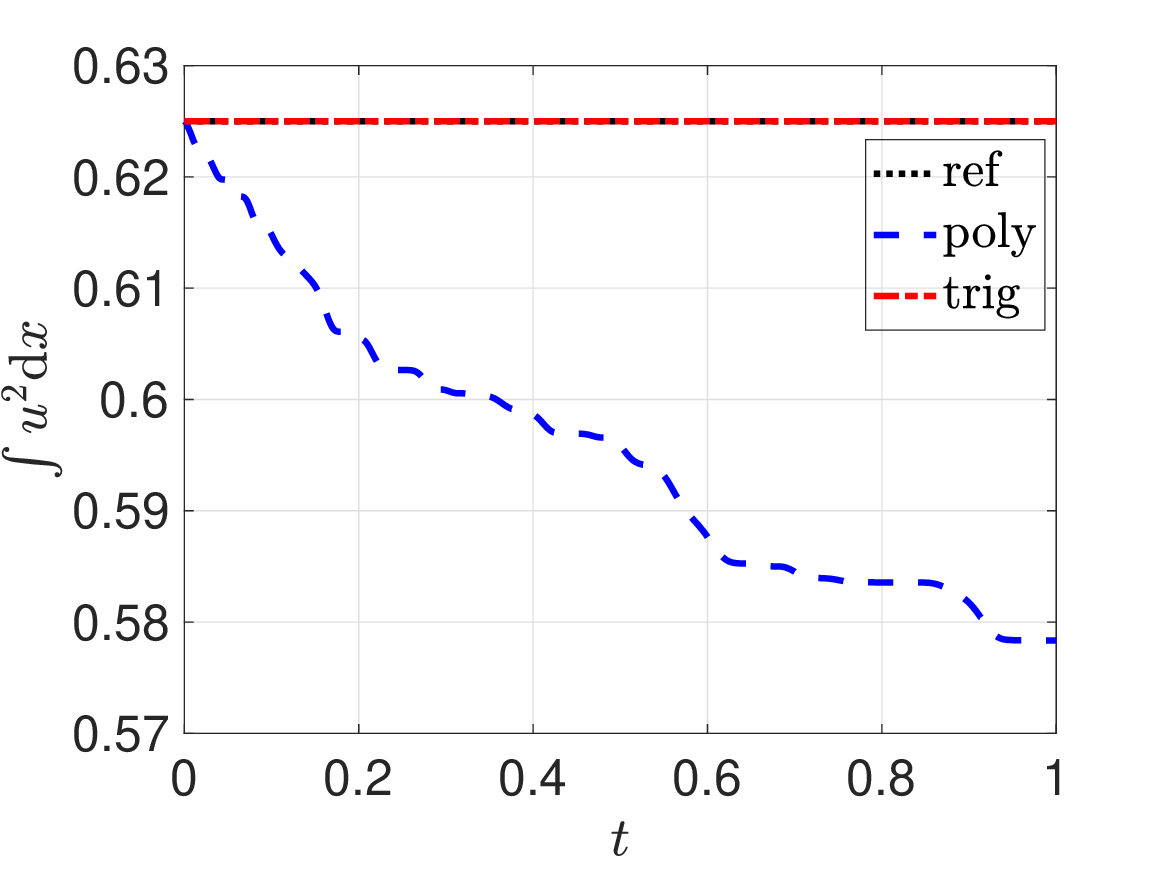} 
    		\caption{Energy over time for $K=81$}
    		\label{fig:linear_energy_K81}
  	\end{subfigure}%
  	\caption{
  	(Numerical) solutions of the linear advection equation \cref{eq:linear_adv} at time $t=1$ as well as mass and energy over time for $K=81$.
	The numerical solutions respectively correspond to the FSBP-SAT method \cref{eq:lin_discr} using the polynomial (``poly") and trigonometric (``trig") approximation space, $\mathcal{P}_{d}$ and $\mathcal{T}_d$, of dimension $K$. 
  	}
  	\label{fig:linear_sol}
\end{figure}  

\cref{fig:linear_sol} provides a comparison between this exact solution and the numerical solutions at time $t=1$. 
The numerical solutions correspond to the FSBP-SAT method \cref{eq:lin_discr} using the polynomial (``poly") and trigonometric (``trig") approximation space, $\mathcal{P}_{d}$ and $\mathcal{T}_d$. 
In \cref{fig:linear_sol_K41,fig:linear_sol_K81} the approximation spaces were chosen to have dimension $K=41$ and $K=81$, respectively. 
In both cases, we can see that the trigonometric approximation space is better suited to represent the solution. 
\cref{fig:linear_mass_K81,fig:linear_energy_K81} illustrate the corresponding mass ($\int u \intd x$) and energy ($\int u^2 \intd x$) profile over time. 
We note that the solution of the trigonometric-based SBP-SAT scheme is closer to the reference solution than the solution of the polynomial-based SBP-SAT scheme. 

\begin{remark} 
	A result similar to the Weierstrass approximation theorem for polynomials states that a continuous and periodic function can be uniformly approximated by the trigonometric polynomials \cite[Theorem 16]{stone1948generalized}.
\end{remark}

\subsection{Inhomogeneous linear advection problems}

Consider the  following inhomogeneous linear advection problem: 
\begin{equation}\label{eq:linear_inhom} 
\begin{aligned} 
	\partial_t u + \partial_x u & = 2u, \quad && 0 < x < \pi, \\ 
	u(x,0) & = 1, \quad && 0 \leq x \leq \pi, \\ 
	u(0,t) & = 1, \quad && t \geq 0,
\end{aligned} 
\end{equation} 
with exact steady state solution $u(x) = e^{2x}$. 
The steady state solution can be expected to be better approximated using an exponential rather than a polynomial approximation space.  

\begin{figure}[tb]
	\centering 
	\begin{subfigure}[b]{0.49\textwidth}
		\includegraphics[width=\textwidth]{%
      		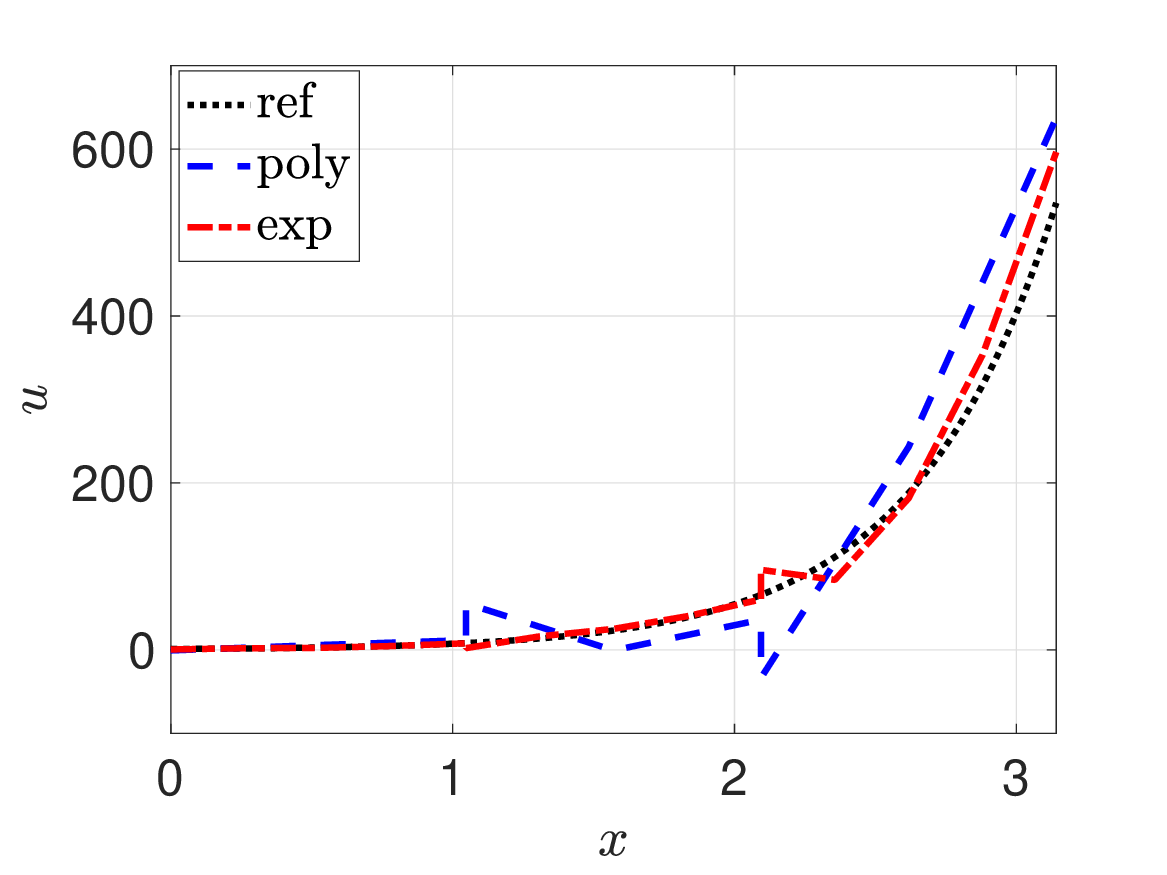} 
    		\caption{Solutions for $I=3$ blocks}
    		\label{fig:linear_source_SAT_K3_I3}
  	\end{subfigure}%
	~
	\begin{subfigure}[b]{0.49\textwidth}
		\includegraphics[width=\textwidth]{%
      		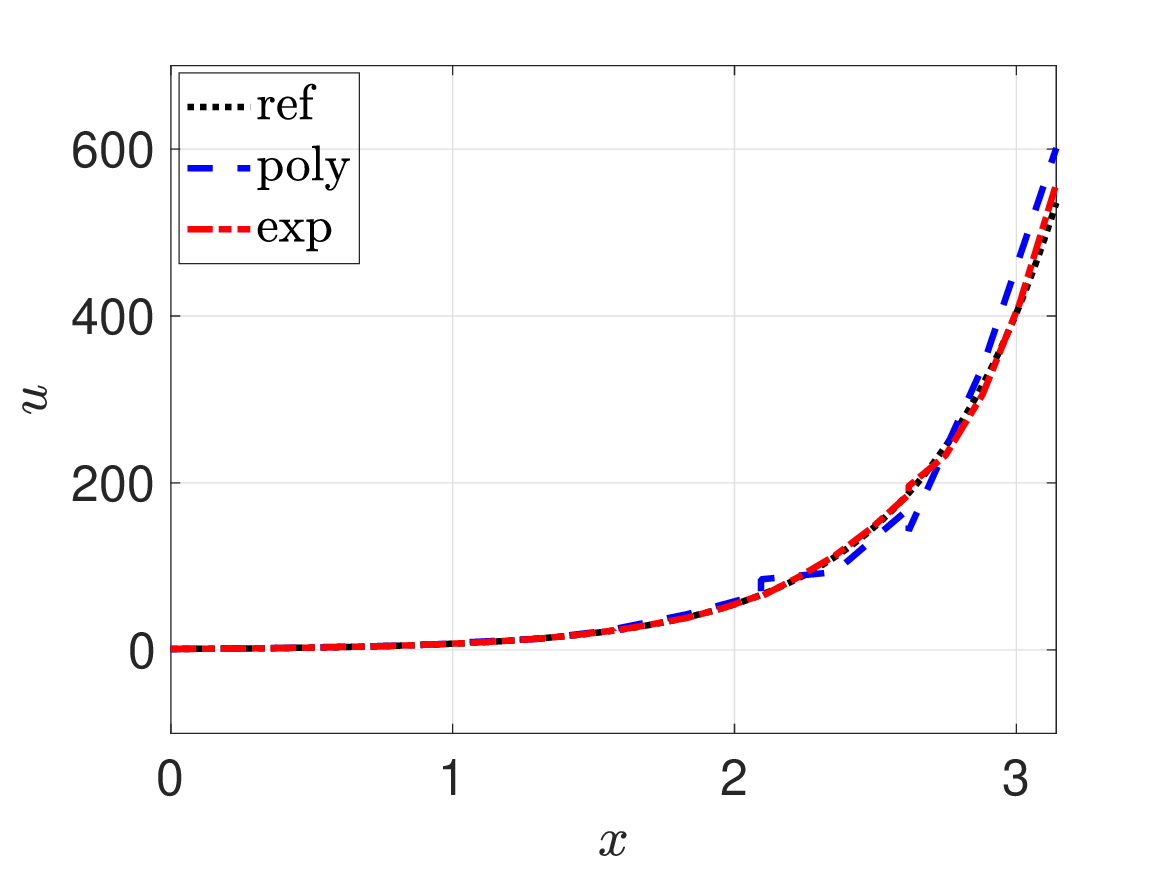} 
    		\caption{Solutions for $I=6$ blocks}
    		\label{fig:linear_source_SAT_K3_I6}
  	\end{subfigure}%
	\\ 
	\begin{subfigure}[b]{0.49\textwidth}
		\includegraphics[width=\textwidth]{%
      		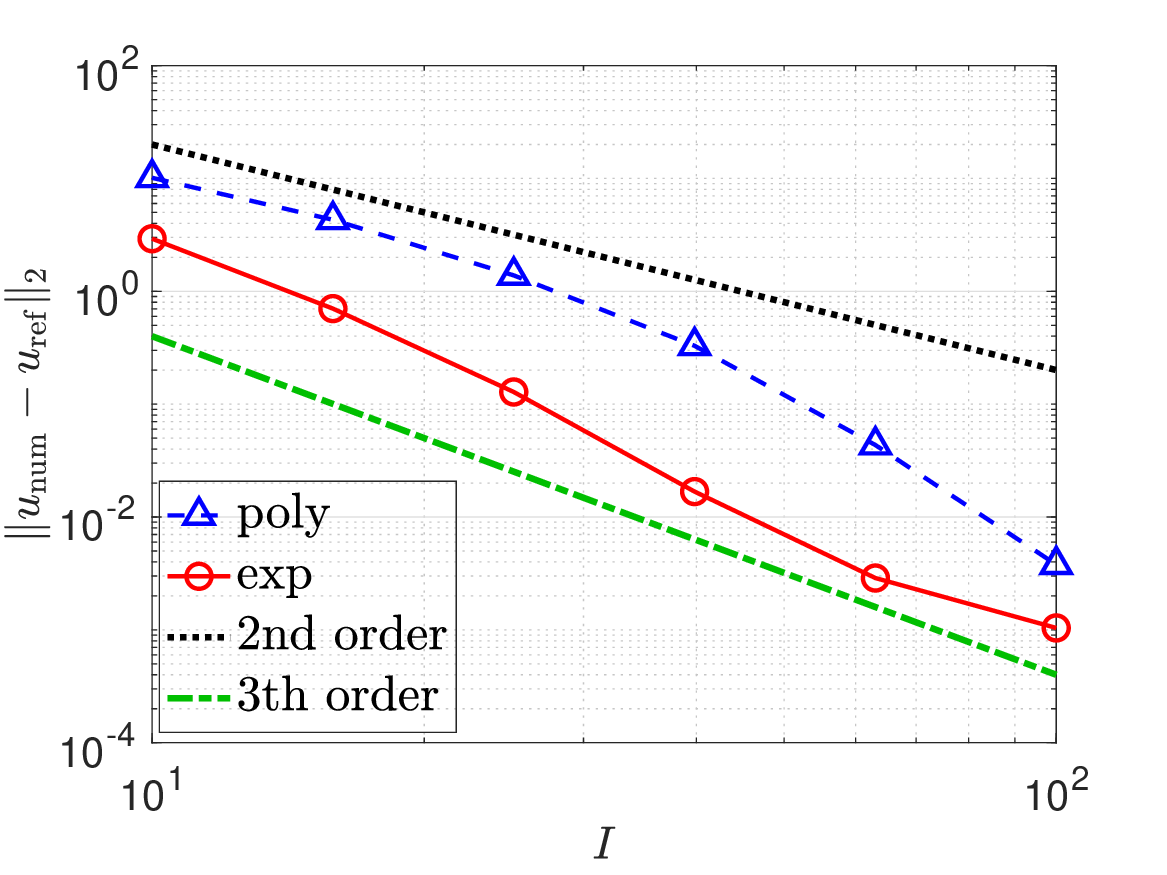} 
    		\caption{$\|\cdot\|_2$-errors}
    		\label{fig:linear_source_SAT_L2error_K3}
  	\end{subfigure}%
	~
  	\begin{subfigure}[b]{0.49\textwidth}
		\includegraphics[width=\textwidth]{%
      		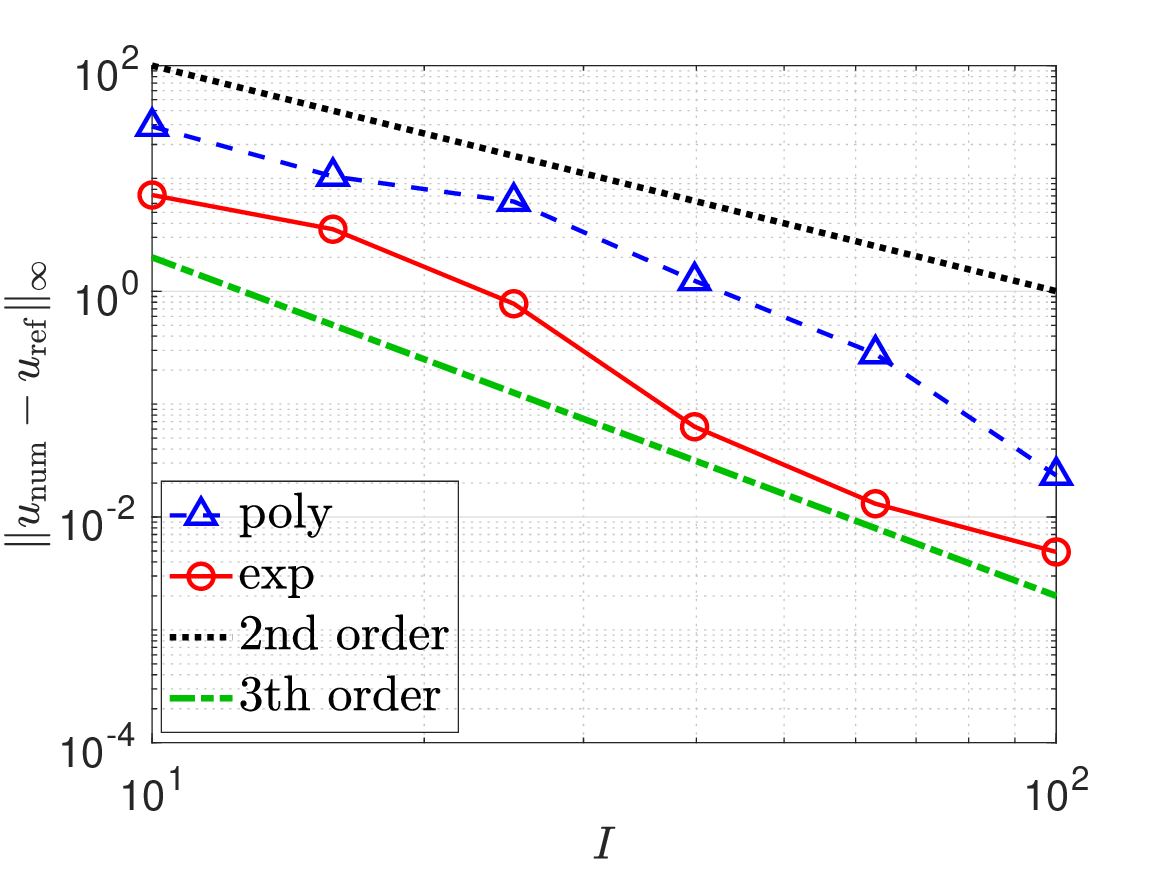} 
    		\caption{$\|\cdot\|_{\infty}$-errors}
    		\label{fig:linear_source_SAT_maxerror_K3}
  	\end{subfigure}%
  	\caption{ 
	(Numerical) solutions and the corresponding errors for the inhomogeneous linear advection problem \cref{eq:linear_inhom} at $t=3.5$. 
	We used a multi-block FSBP-SAT scheme with a polynomial and exponential approximation space, $\mathcal{P}_2 = \Span\{1,x,x^2\}$ and $\mathcal{E}_2 = \Span\{1,x,e^x\}$, respectively.  
  	}
  	\label{fig:linear_inhom}
\end{figure}  

\cref{fig:linear_inhom} shows the (numerical) solutions (\cref{fig:linear_source_SAT_K3_I3,fig:linear_source_SAT_K3_I6}) and their errors (\cref{fig:linear_source_SAT_L2error_K3,fig:linear_source_SAT_maxerror_K3}) at time $t=3.5$.  
The numerical solutions were computed using a multi-block FSBP-SAT method with a three-dimensional polynomial and exponential approximation space $\mathcal{P}_2 = \Span\{1,x,x^2\}$ and $\mathcal{E}_2 = \Span\{1,x,e^x\}$, respectively. 
\cref{fig:linear_source_SAT_K3_I3,fig:linear_source_SAT_K3_I6} illustrate the results using $I=3$ and $I=6$ uniform blocks. 
The polynomial SBP operator was constructed using three Gauss--Lobatto points and weights, while the exponential FSBP operator was described in \cref{sub:exp_FSBP}. 
We find that the exponential approximation space $\mathcal{E}_2$ yields more accurate results than the usual polynomial approximation space $\mathcal{P}_2$. 
In accordance with this, \cref{fig:linear_source_SAT_L2error_K3,fig:linear_source_SAT_maxerror_K3} provide the $\|\cdot\|_2$- and $\|\cdot\|_{\infty}$-errors of both FSBP-SAT methods for an increasing number of uniform blocks, $I$. 
While the convergence rate of the multi-block FSBP-SAT methods seems to be roughly the same for both approximation spaces, the individual error levels are found to be smaller for the $\mathcal{E}_2$-based SBP operator. 
It also seems that the difference between the numerical solutions corresponding to the exponential and polynomial approximation space decreases for large numbers of blocks, $I$. 
This might be explained by the exponential approximation space mostly being of advantage in blocks where the solution rapidly changes, and as the total number of blocks increases the proportion of blocks within which the solution rapidly changes becomes smaller. 

\begin{figure}[tb]
	\centering 
	\begin{subfigure}[b]{0.49\textwidth}
		\includegraphics[width=\textwidth]{%
      		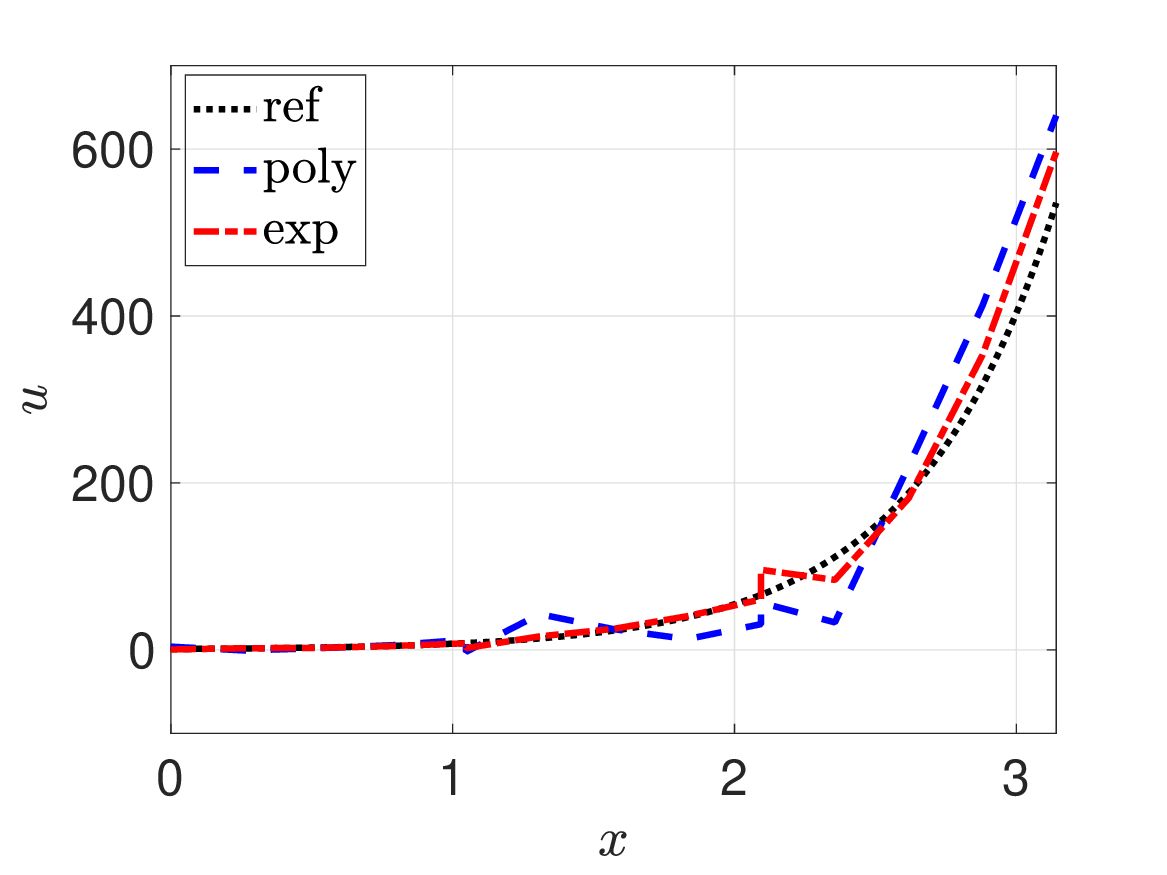} 
    		\caption{Solutions for $I=3$ blocks}
    		\label{fig:linear_source_SAT2_K3_I3}
  	\end{subfigure}%
	~
	\begin{subfigure}[b]{0.49\textwidth}
		\includegraphics[width=\textwidth]{%
      		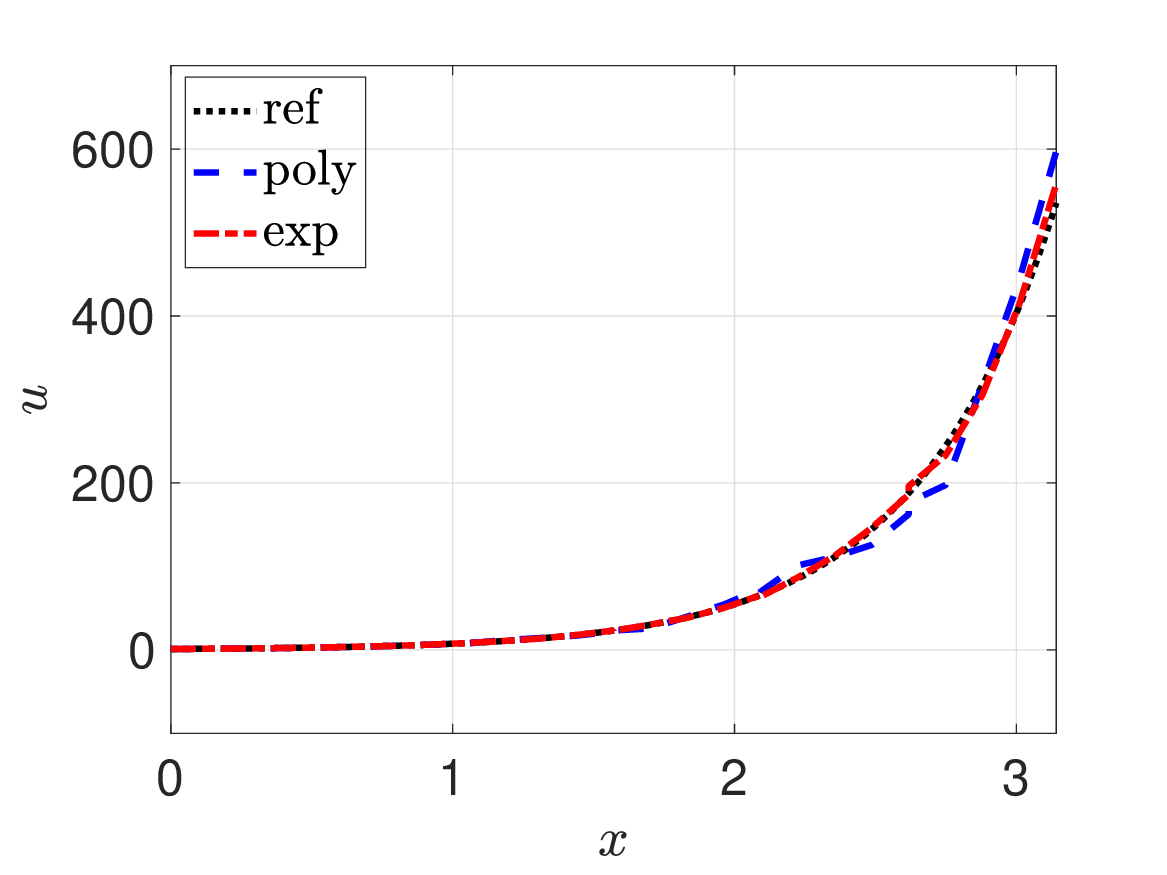} 
    		\caption{Solutions for $I=6$ blocks}
    		\label{fig:linear_source_SAT2_K3_I6}
  	\end{subfigure}%
	\\ 
	\begin{subfigure}[b]{0.49\textwidth}
		\includegraphics[width=\textwidth]{%
      		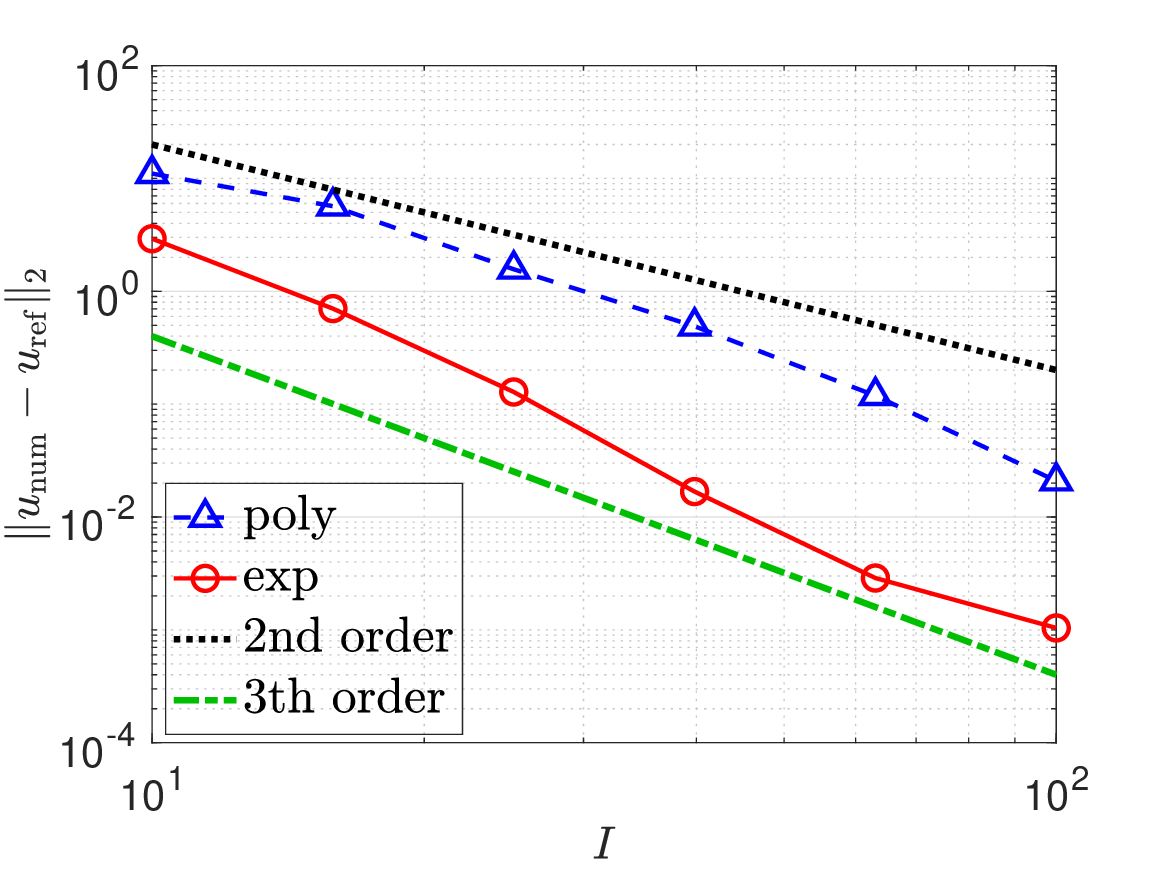} 
    		\caption{$\|\cdot\|_2$-errors}
    		\label{fig:linear_source_SAT_L2error2_K3}
  	\end{subfigure}%
	~
  	\begin{subfigure}[b]{0.49\textwidth}
		\includegraphics[width=\textwidth]{%
      		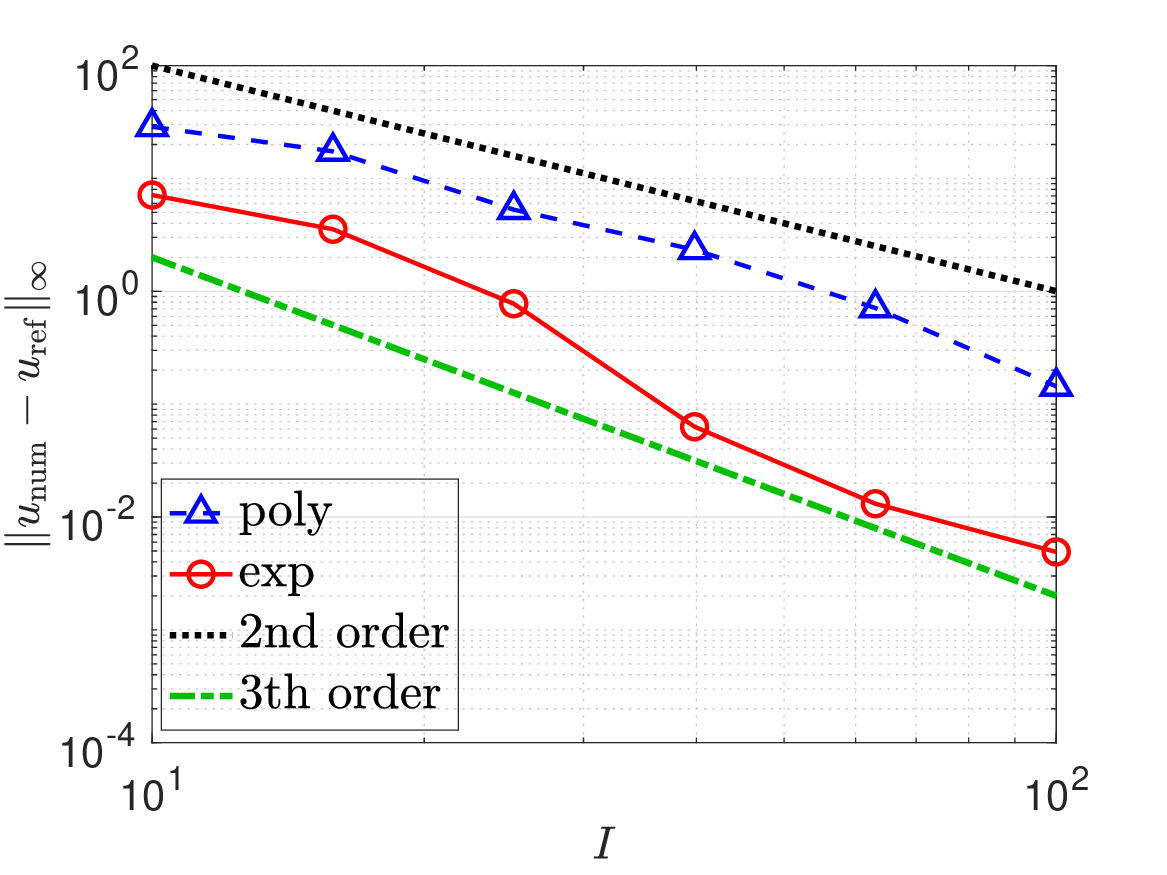} 
    		\caption{$\|\cdot\|_{\infty}$-errors}
    		\label{fig:linear_source_SAT_maxerror2_K3}
  	\end{subfigure}%
  	\caption{ 
	(Numerical) solutions and the corresponding errors for the inhomogeneous linear advection problem \cref{eq:linear_inhom} at $t=3.5$. 
	We used a multi-block FSBP-SAT scheme with a polynomial and exponential approximation space, $\mathcal{P}_2 = \Span\{1,x,x^2\}$ and $\mathcal{E}_2 = \Span\{1,x,e^x\}$, respectively. 
	Both schemes use the same grid points.
  	}
  	\label{fig:linear_inhom2}
\end{figure}

One might wonder if the increased accuracy of the exponential FSBP operator compared to the polynomial SBP operator observed in \cref{fig:linear_inhom} is due to the difference in the function spaces or can be explained by the exponential FSBP operator using two more grid points per block. 
To address this question, we repeat the above test using a polynomial SBP operator that is exact for $\mathcal{P}_2 = \Span\{1,x,x^2\}$ and uses the same grid points as the exponential FSBP operator. 
The corresponding results can be found in \cref{fig:linear_inhom2} and demonstrate that the FSBP operator again yields more accurate results. 

\begin{remark} 
	The Stone--Weierstrass approximation theorem \cite{stone1948generalized}---a generalization of the Weierstrass approximation theorem for polynomials---ensures uniform $p$-convergence for a large class of non-polynomial function spaces to any continuous function. 
	Moreover, $h$-convergence for some non-polynomial function spaces was addressed, for instance, in \cite{yuan2006discontinuous,yang2016short}. 
	While the exponential function space considered here was not explicitly addressed,  we numerically observe $h$-convergence in \cref{fig:linear_inhom,fig:linear_inhom2}. 
\end{remark}

\subsection{Burgers' equation} 

Consider the nonlinear inviscid Burgers' equation 
\begin{equation}\label{eq:Burgers} 
\begin{aligned} 
	\partial_t u + \partial_x \left( \frac{u^2}{2} \right) & = 0, \quad && 0 < x < 1, \\ 
	u(x,0) & = 1 + \frac{1}{2} \sin(4 \pi x)^3 + \frac{1}{4} \cos(4 \pi x)^5, \quad && 0 \leq x \leq 1, \\ 
	u(0,t) & = u(1,t), \quad && t \geq 0, 
\end{aligned}
\end{equation} 
with periodic boundary conditions and smooth but rapidly varying initial data.
We choose the initial data so that the solution of \cref{eq:Burgers} includes gradients of large magnitude, which can be expected to be better approximated by exponential rather than polynomial approximation spaces.  
Furthermore, we again consider a uniform multi-block FSBP-SAT semi-discretization.
Recall that an energy-stable skew-symmetric FSBP-SAT semi-discretization of \cref{eq:Burgers} is given by 
\begin{equation}\label{eq:Burgers_discr}
	\mathbf{u}_t^{(i)} 
		+ \frac{1}{3} D U^{(i)} \mathbf{u}^{(i)} + \frac{1}{3} U^{(i)}D  \mathbf{u}^{(i)} =  P^{-1}\mathbb{S}^{(i)}, 
		\quad i=1,\dots,I.
\end{equation} 
Here, $I$ denotes the number of uniform blocks, $U^{(i)} = \diag(\mathbf{u})^{(i)}$, and the SAT $\mathbb{S}^{(i)}$ is given by  
\begin{equation}\label{eq:Burgers_SAT}
	\mathbb{S}^{(i)} = [\mathbb{S}_1^{(i)},0,\dots,0]^T, \quad 
	\mathbb{S}_1^{(i)} = - \frac{2}{3} u_1^{(i)} (u_1^{(i)} - g_L^{(i)})
\end{equation} 
with $g_L^{(i)} = u_N^{(I)}$ for $i=1$ (to weakly enforce the periodic boundary conditions) and $g_L^{(i)} = u_N^{(i-1)}$ (to weakly couple neighboring blocks) otherwise. 

\begin{figure}[tb]
	\centering 
	\begin{subfigure}[b]{0.49\textwidth}
		\includegraphics[width=\textwidth]{%
      		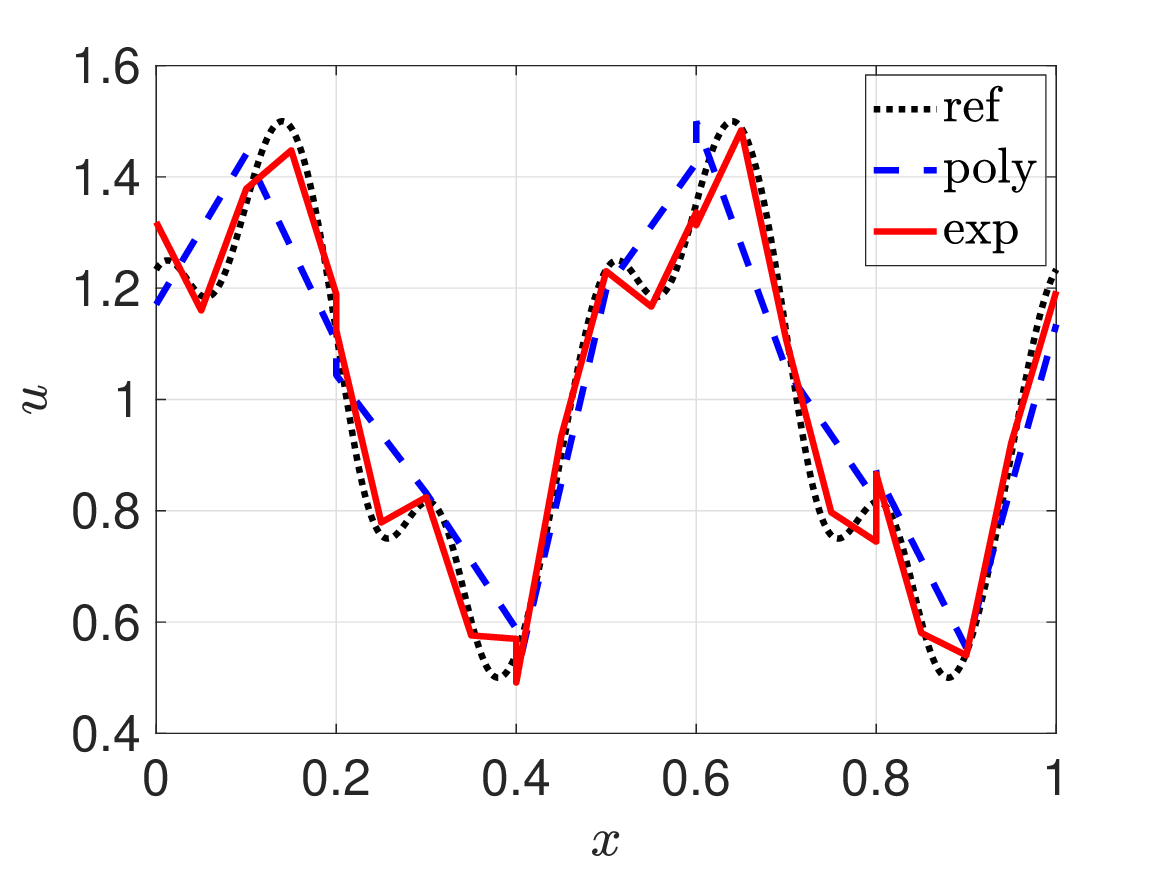} 
    		\caption{Solutions for $I=5$ blocks}
    		\label{fig:Burgers_sol_I5}
  	\end{subfigure}%
	~
	\begin{subfigure}[b]{0.49\textwidth}
		\includegraphics[width=\textwidth]{%
      		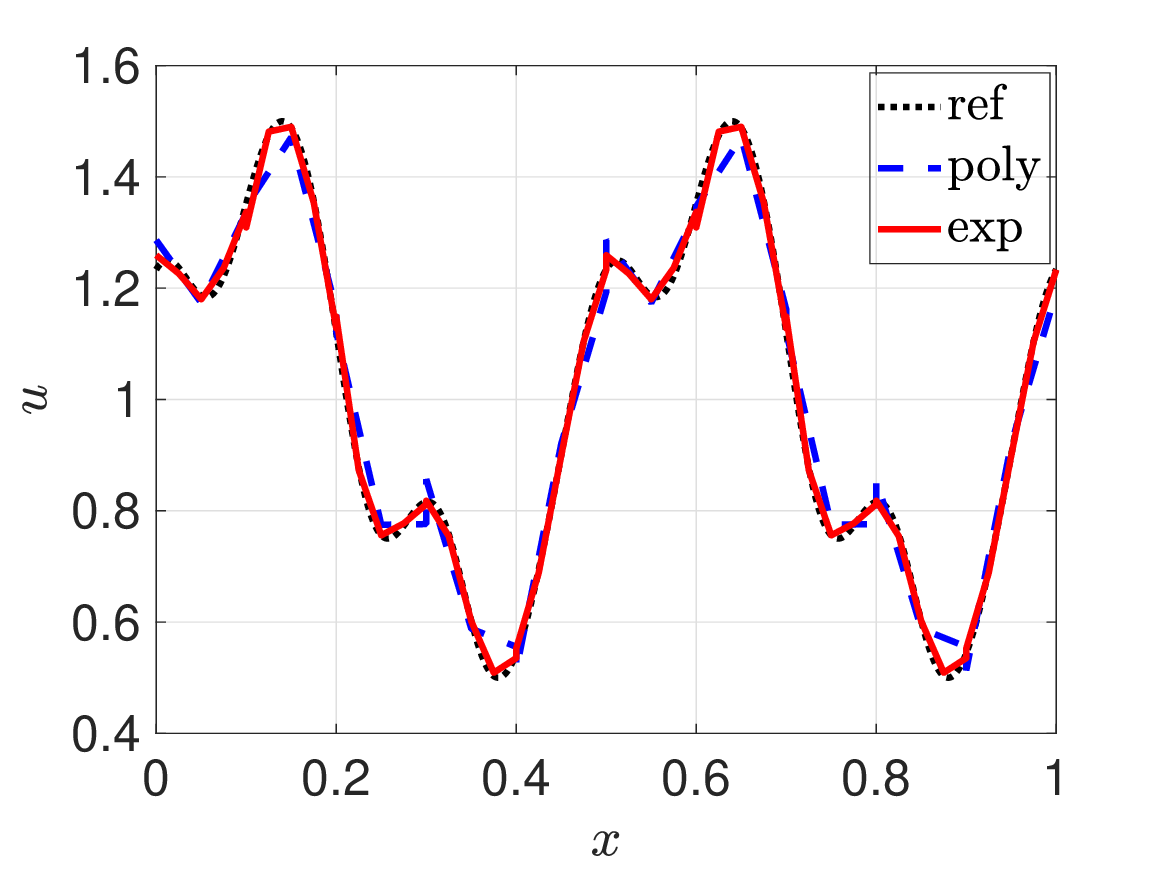} 
    		\caption{Solutions for $I=10$ blocks}
    		\label{fig:Burgers_sol_I10}
  	\end{subfigure}%
	\\ 
	\begin{subfigure}[b]{0.49\textwidth}
		\includegraphics[width=\textwidth]{%
      		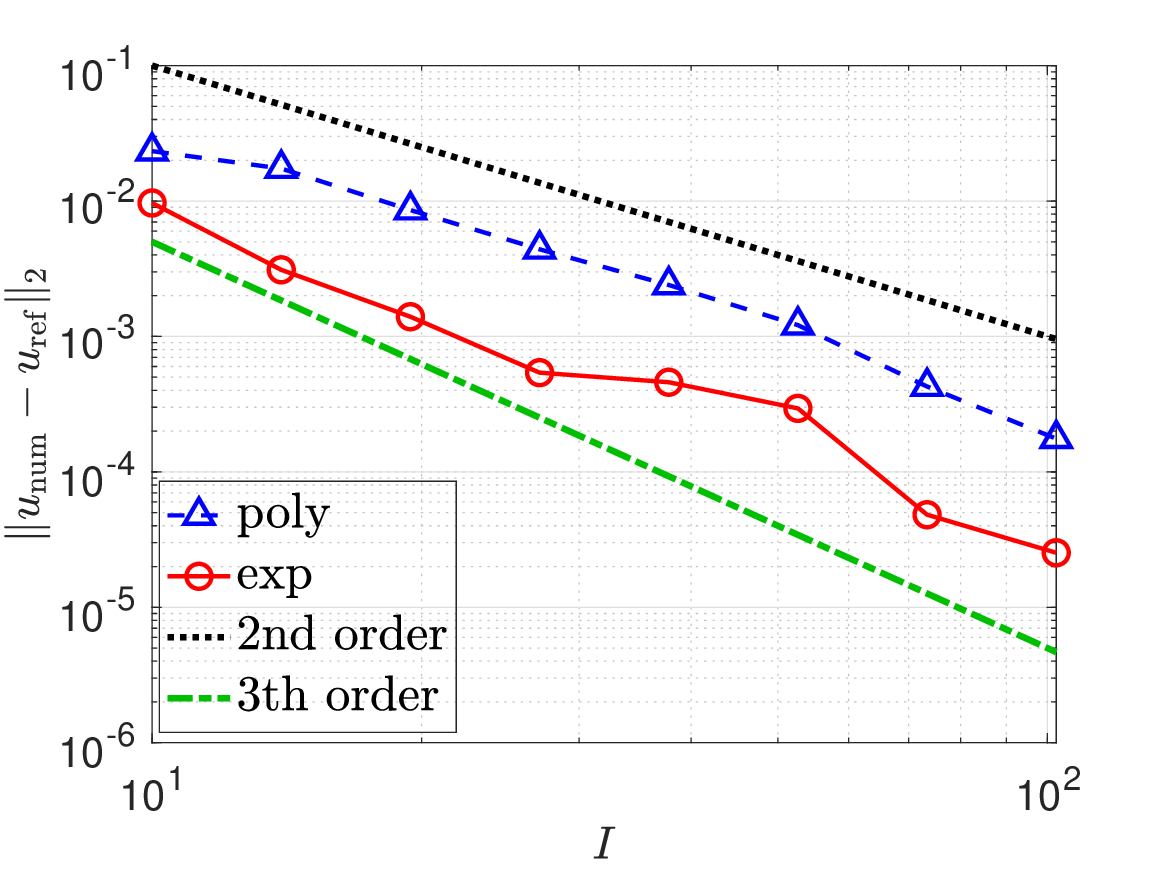} 
    		\caption{$\|\cdot\|_2$-errors}
    		\label{fig:Burgers_L2errors}
  	\end{subfigure}%
	~
  	\begin{subfigure}[b]{0.49\textwidth}
		\includegraphics[width=\textwidth]{%
      		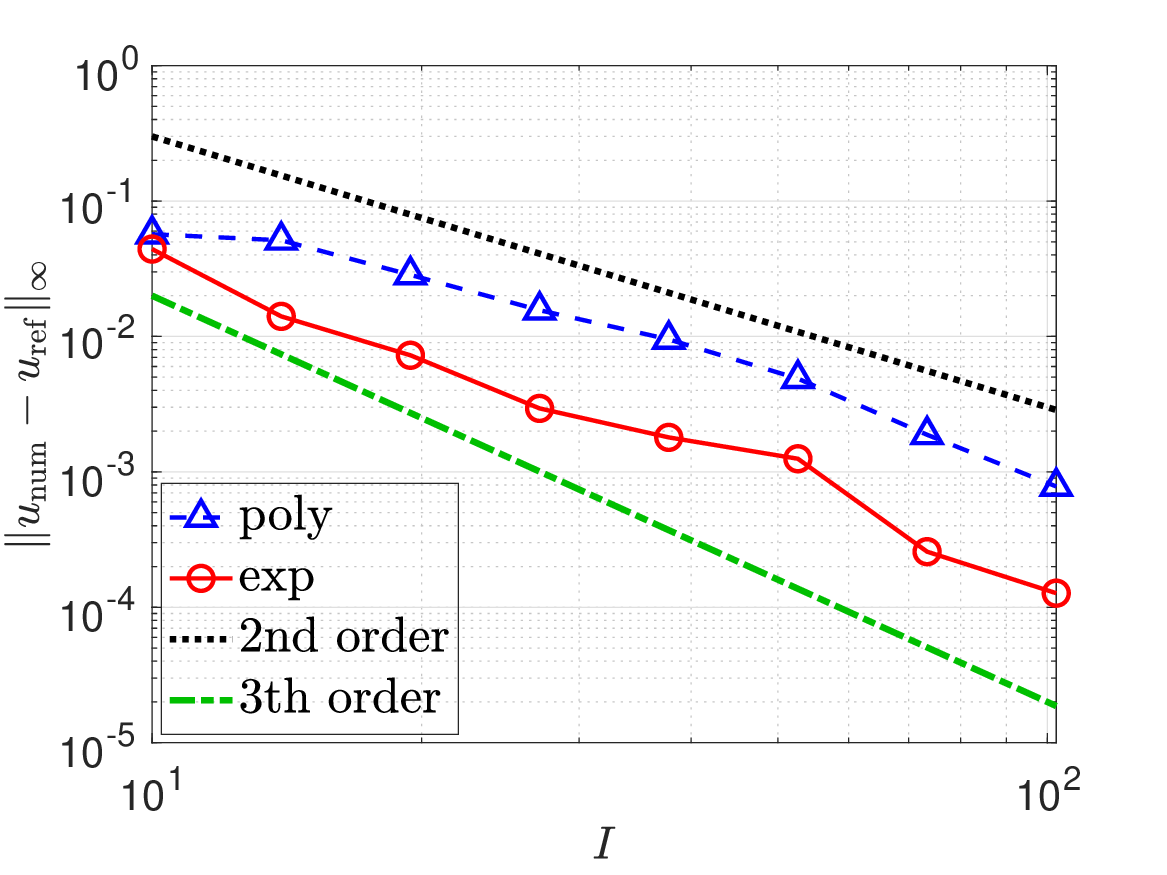} 
    		\caption{$\|\cdot\|_{\infty}$-errors}
    		\label{fig:Burgers_maxerrors}
  	\end{subfigure}%
  	\caption{ 
	(Numerical) solutions and the corresponding errors for Burgers' equation \cref{eq:Burgers} at $t=0.01$. 
	We used a multi-block FSBP-SAT semi-discretization \cref{eq:Burgers_discr} with a polynomial and exponential approximation space, $\mathcal{P}_2 = \Span\{1,x,x^2\}$ and $\mathcal{E}_2 = \Span\{1,x,e^x\}$, respectively.  
  	}
  	\label{fig:Burgers}
\end{figure}  

\cref{fig:Burgers} shows the (numerical) solutions (\cref{fig:Burgers_sol_I5,fig:Burgers_sol_I10}) and their errors (\cref{fig:Burgers_L2errors,fig:Burgers_maxerrors}) at time $t=0.01$.  
The reference solution (``ref") was computed using characteristic tracing. 
The numerical solutions were computed using a multi-block FSBP-SAT method with a three-dimensional polynomial and exponential approximation space $\mathcal{P}_2 = \Span\{1,x,x^2\}$ and $\mathcal{E}_2 = \Span\{1,x,e^x\}$, respectively. 
\cref{fig:Burgers_sol_I5,fig:Burgers_sol_I10} illustrate the results using $I=5$ and $I=10$ uniform blocks. 
We again find that the exponential approximation space $\mathcal{E}_2$ yields more accurate results than the usual polynomial approximation space $\mathcal{P}_2$. 
In accordance with this, \cref{fig:Burgers_L2errors,fig:Burgers_maxerrors} provide the $\|\cdot\|_2$- and $\|\cdot\|_{\infty}$-errors of both FSBP-SAT methods for an increasing number of uniform blocks, $I$. 
While the convergence rate of the multi-block FSBP-SAT methods seems to be roughly the same for both approximation spaces, the individual error levels are all found to be smaller for the $\mathcal{E}_2$-based SBP operator. 
In this case, the error difference is preserved on fine grids. 
We suspect that this is because the solution varies rapidly throughout the whole domain, and the $\mathcal{E}_2$-based SBP operators are therefore particularly advantageous.

\subsection{A discontinuous problem} 

Consider the inviscid Burgers' equation 
\begin{equation}\label{eq:Burgers_disc} 
\begin{aligned} 
	\partial_t u + \partial_x \left( \frac{u^2}{2} \right) & = 0, \quad && 0 < x < 1, \\ 
	u(x,0) & = 1 + \frac{3}{4} \sin(2 \pi x), \quad && 0 \leq x \leq 1, \\ 
	u(0,t) & = u(1,t), \quad && t \geq 0, 
\end{aligned}
\end{equation} 
with periodic boundary conditions and smooth initial data. 
The solution of \cref{eq:Burgers_disc} is smooth until the breaking time $t_b = 2/3\pi \approx 0.212$, after which the solution contains a shock discontinuity. 
Also note that the smooth solution features increasingly sharp gradients as $t \to t_b$. 

\begin{figure}[tb]
	\centering 
	\begin{subfigure}[b]{0.49\textwidth}
		\includegraphics[width=\textwidth]{%
      		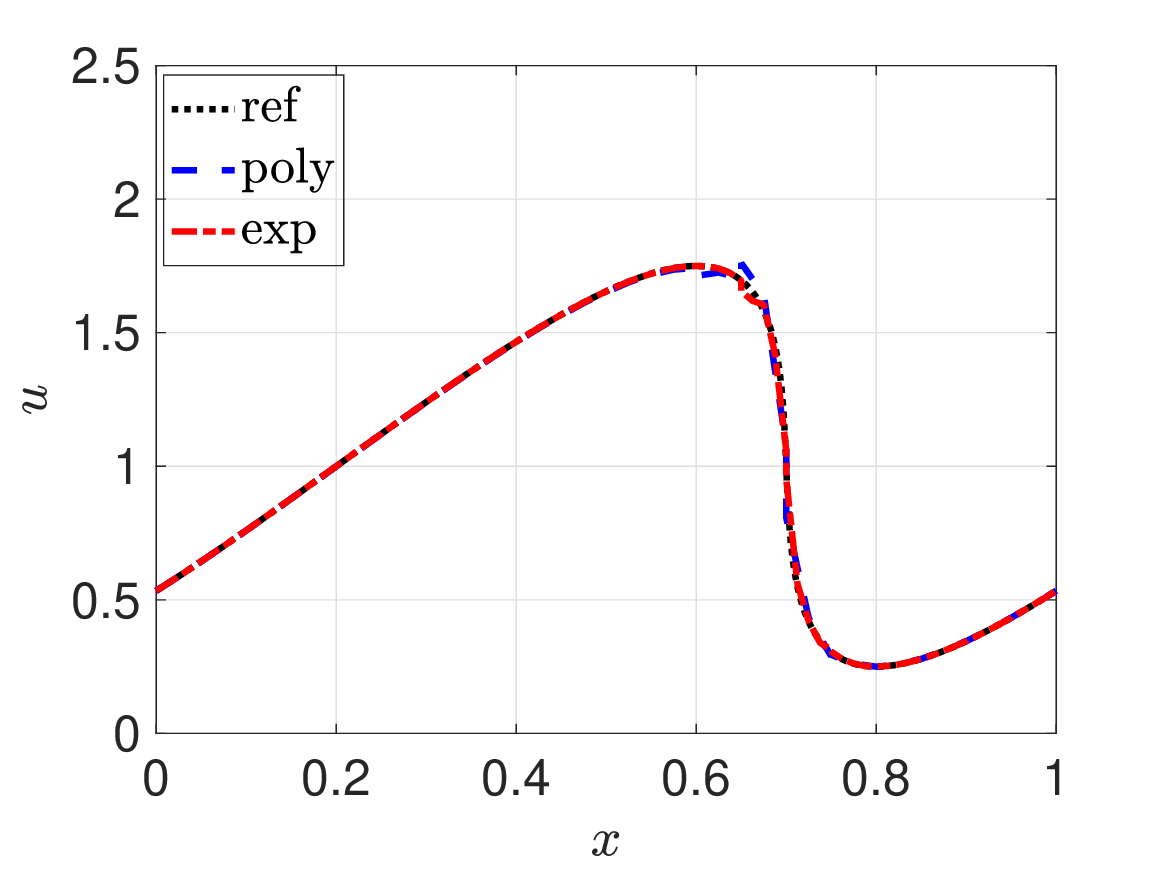} 
    		\caption{Solutions at $t = 0.2$}
    		\label{fig:Burgers_disc_K3_I20_T200}
  	\end{subfigure}%
	~ 
	\begin{subfigure}[b]{0.49\textwidth}
		\includegraphics[width=\textwidth]{%
      		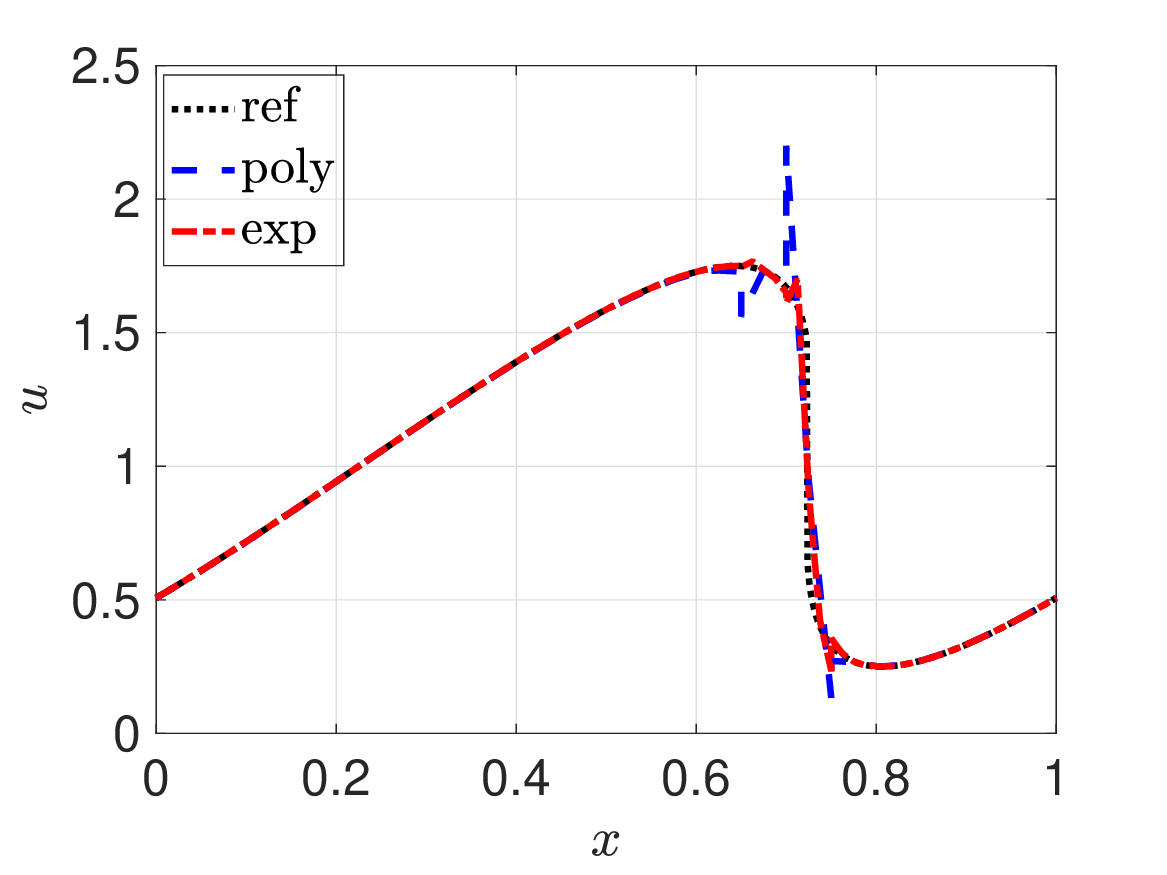} 
    		\caption{Solutions at $t = 0.225$}
    		\label{fig:Burgers_disc_K3_I20_T225}
  	\end{subfigure}%
	\\
	\begin{subfigure}[b]{0.49\textwidth}
		\includegraphics[width=\textwidth]{%
      		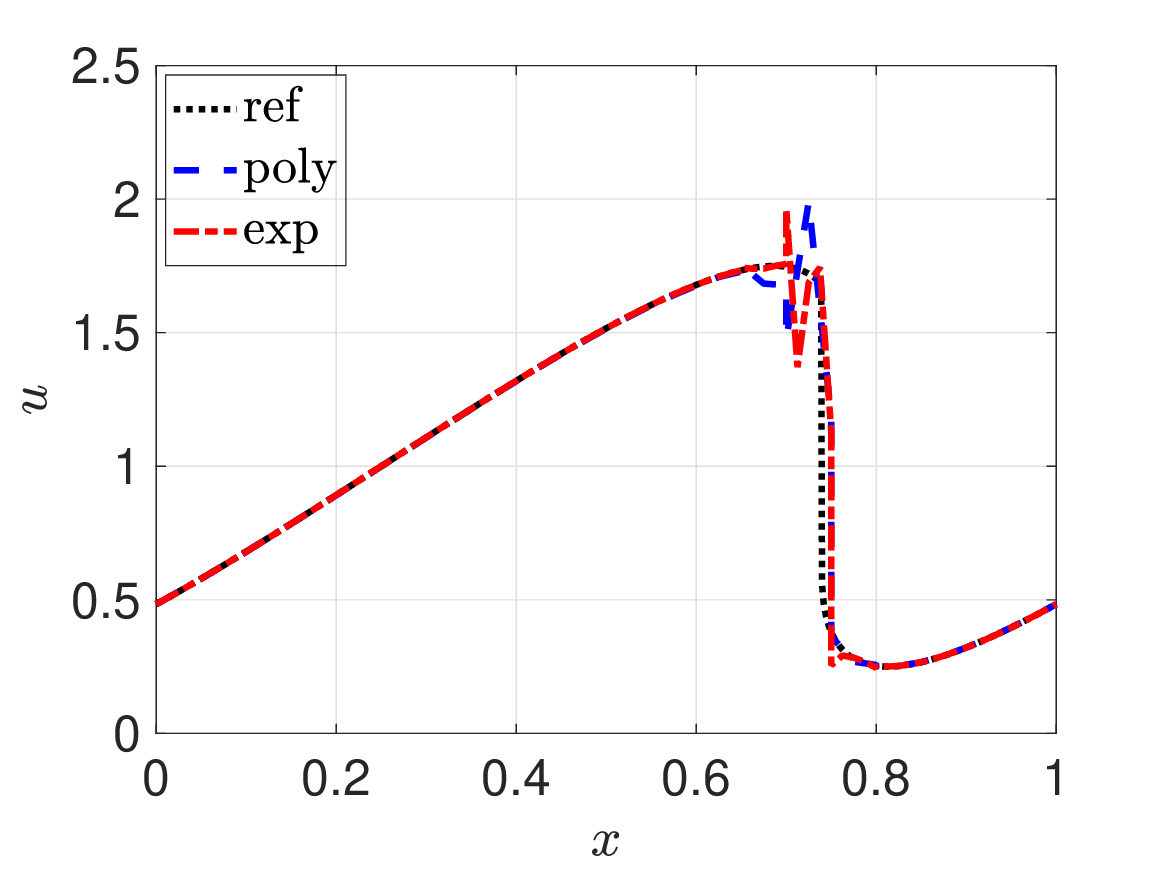} 
    		\caption{Solutions at $t = 0.25$}
    		\label{fig:Burgers_disc_K3_I20_T250}
  	\end{subfigure}%
	~ 
	\begin{subfigure}[b]{0.49\textwidth}
		\includegraphics[width=\textwidth]{%
      		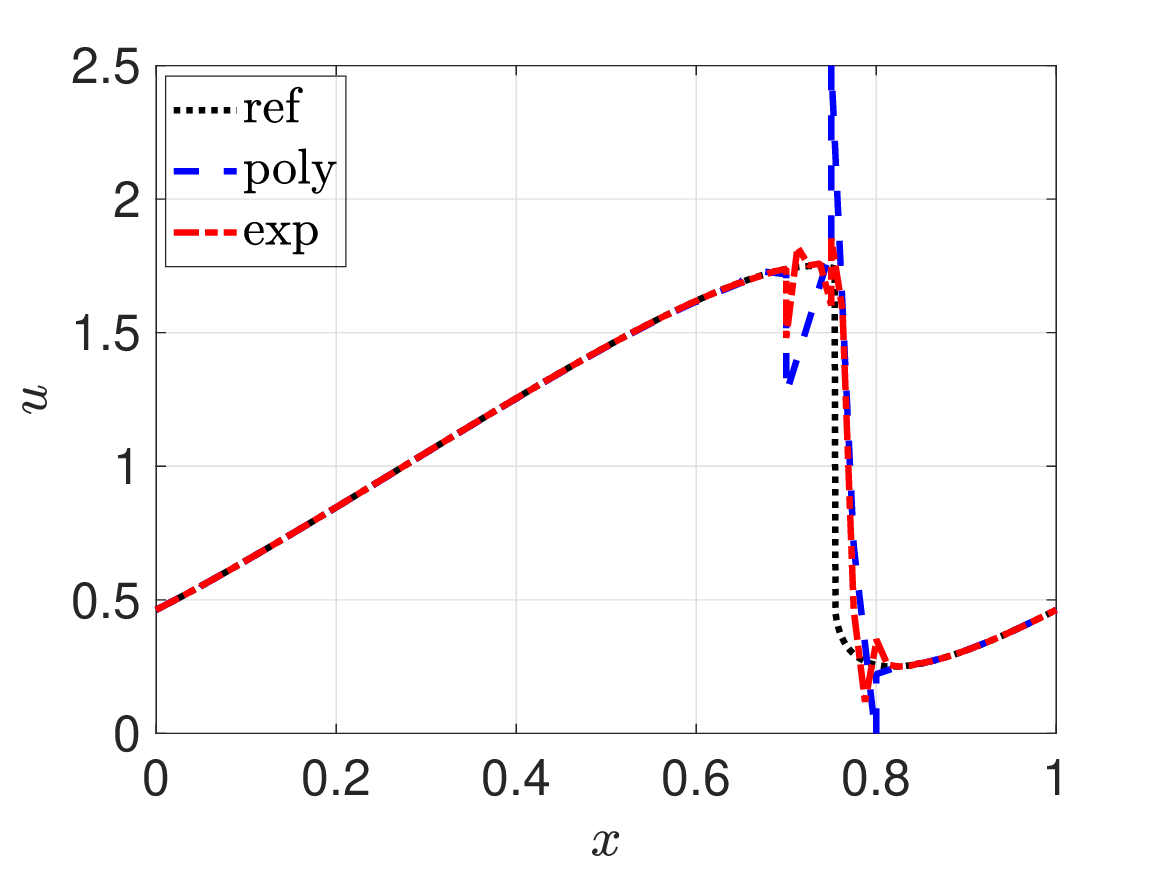} 
    		\caption{Solutions at $t = 0.275$}
    		\label{fig:Burgers_disc_K3_I20_T275}
  	\end{subfigure}%
  	\caption{ 
	(Numerical) solutions for Burgers' equation \cref{eq:Burgers_disc}. 
	The exact solution features a shock discontunity starting at $t_b \approx 0.212$. 
	We used a multi-block FSBP-SAT semi-discretization \cref{eq:Burgers_discr} with $I=20$ blocks and a polynomial and exponential approximation space, $\mathcal{P}_2 = \Span\{1,x,x^2\}$ and $\mathcal{E}_2 = \Span\{1,x,e^x\}$, respectively.  
  	}
  	\label{fig:Burgers_disc}
\end{figure}

\cref{fig:Burgers_disc} shows the (numerical) solutions at different times close to the breaking time. 
We observe that using the exponential function space allows us to better resolve the steep gradients of the solution shortly before the breaking time. 
However, once the shock discontinuity has formed (see \cref{fig:Burgers_disc_K3_I20_T250,fig:Burgers_disc_K3_I20_T275}), also the numerical solution using the exponential function space shows spurious oscillations. 
Future work will include limiting \cite{cockburn1989tvb,zhang2011maximum}, artificial dissipation \cite{mattsson2004stable,ranocha2018stability}, and other shock capturing procedures for non-polynomial approximation spaces and FSBP operators.  
\section{Concluding thoughts} 
\label{sec:summary} 

In this work, we introduced the concept of FSBP operators, which is a generalization of the existing polynomial-based SBP operators to general approximation spaces. 
We demonstrated that using, for instance, trigonometric or exponential approximation spaces can be of advantage when we have some prior knowledge about the behavior (periodic, oscillatory, varies rapidly) of the otherwise unknown solution.  
Furthermore, we showed that most of the mimetic properties for polynomial-based SBP operators carry over to the more general class of FSBP operators. 
The same holds true for the construction of FSBP operators, which can be done using the same strategies as for polynomial-based SBP operators. 
At the same time, we stress that for FSBP operators to mimic conservation, the associated approximation space needs to include constants, which is naturally satisfied for polynomial-based SBP operators. 
Our findings imply that the concept of SBP operators can be applied to a larger class of methods than currently known. 
Such methods might include exponentially fitted schemes for singular perturbation problems, discontinuous Galerkin methods based on non-polynomial approximation spaces, and certain ENO and WENO reconstructions. 
Since the first draft of the present manuscript, we have demonstrated in \cite{glaubitz2022energy} how FSBP operators can be used to construct provable energy-stable global radial basis function methods. 
In future work, the extension of the proposed FSBP operators to multiple dimensions as well as an investigation of their CFL limits and dispersion properties similar to \cite{gassner2011comparison} will be considered. 
Also limiting \cite{cockburn1989tvb,zhang2011maximum}, artificial dissipation \cite{mattsson2004stable,ranocha2018stability}, and other shock capturing procedures for FSBP operators in combination with their application to discontinuous problems will be included in future efforts. 

\appendix 
\section{Error minimization for $\sigma=1$}  
\label{app:error}

Let $u$ and $v$ denote the solution of the linear advection equation \cref{eq:linear} with the same initial condition but different boundary data $u(t,x_L) = g$ and $v(t,x_L) = \tilde{g}(t)$, respectively. 
Furthermore, we denote their difference by $d = u - v$ and the difference in the boundary data by $\delta = g - \tilde{g}$. 
Assume that the semi-discretizations $\mathbf{u}, \mathbf{v}$ of $u,v$ are obtained by an FSBP-SAT method. 
That is, they satisfy \cref{eq:lin_discr}. 
Consequently, the error between the two numerical solutions, $\mathbf{e} = \mathbf{u} - \mathbf{v}$, satisfies 
\begin{equation} 
	\mathbf{e}_t  +a D \mathbf{e} = P^{-1} \left( \mathbb{S} - \tilde{\mathbb{S}} \right),
\end{equation}
where 
\begin{equation} 
\begin{aligned} 
	\mathbb{S} 
		& = [\mathbb{S}_1,0,\dots,0]^T, \quad 
		&& \mathbb{S}_1 = - \sigma a (u_1-g), \\ 
	\tilde{\mathbb{S}} 
		& = [\tilde{\mathbb{S}}_1,0,\dots,0]^T, \quad 
		&& \tilde{\mathbb{S}}_1 = - \sigma a (v_1-\tilde{g}).
\end{aligned}
\end{equation} 
We can follow the same arguments as in \cref{sub:linear_energy} (also see \cite[page 11]{aalund2016provably}) and arrive at  
\begin{equation} 
	\frac{\d}{\d t} \| \mathbf{e} \|_P^2
		\leq a \delta^2 \left( \frac{\sigma}{2\sigma-1} \right)
\end{equation} 
for $\sigma > 1/2$. 
That is, the rate of change of the difference between $\mathbf{u}$ and $\mathbf{v}$ can be bounded in terms of the boundary data error $\delta = g - \tilde{g}$ and the term $\sigma/(2\sigma-1)$. 
Finally, this term is minimized for $\sigma = 1$. 
 
\section{Outlook on nonlinear problems and entropy stability} 
\label{app:outlook} 

In the classical SBP theory, it is well known that provable energy-stability for nonlinear equations (or variable coefficients) relies on splitting formulations of the equation \cite{fernandez2014review,nordstrom2006conservative,offner2019error}. 
The same holds true for FSBP operators and the corresponding FSBP-SAT method.
Indeed, by now, we have observed several times that most of the well-known results for polynomial-based SBP operators also hold for FSBP operators. 
A natural extension might be now to combine FSBP operators with the popular flux differencing approach \cite{carpenter2014entropy,offner2019error,abgrall2019reinterpretation,chen2020review}. 
The main idea behind this technique is to apply an entropy-conservative flux in the sense of Tadmor \cite{tadmor1987numerical} for the internal point discretization. 
When starting with classical SBP operators, one can manipulate the internal discretization by constructing telescoping operators \cite{fisher2013high_2, fisher2013discretely}. 
Following this procedure, one would analogously obtain entropy-stable schemes using FSBP-operators. 
The results transform naturally to the more general setting with only minor modifications.
In our numerical Section \ref{sec:num_tests}, we have demonstrated this.
A skew-symmetric formulation has been used for the Burgers' equation resulting in an entropy dissipative scheme.
A more detailed demonstration will be provided in future work.   

\section*{Acknowledgements}
This research was supported by AFOSR \#F9550-18-1-0316, ONR MURI \#N00014-20-1-2595, Vetenskapsr\r{a}det Sweden grant 2018-05084 VR and 2021-05484, the Swedish e-Science Research Center (SeRC), and the Gutenberg Research College, JGU Mainz.

\small
\bibliographystyle{siamplain}
\bibliography{references}

\end{document}